\documentclass[11pt]{article}

\usepackage{amsmath}
\usepackage{amssymb}  
\usepackage{epsfig}
\usepackage{amsthm}   
\usepackage[mathcal]{eucal}
\usepackage{mathrsfs}
\usepackage{url}

\usepackage{color}

\usepackage{mathtools}

\usepackage{geometry}
\newcommand{\bigo}{\mathcal{O}}
\newcommand{\R}{\mathbb{R}}

\newcommand{\Abs}[1]{\left|#1\right|}
\newcommand{\abs}[1]{|#1|}
\newcommand{\inner}[2]{\langle{#1},{#2}\rangle}

\newcommand{\norm}[1]{\|#1\|}
\newcommand{\Norm}[1]{\left\|#1\right\|}

\newcommand{\tos}{\rightrightarrows}

\newcommand{\HH}{\mathcal{H}}

\DeclareMathOperator{\dom}{D}

\newcommand{\poubelle}[1]{}

\newtheorem{theorem}{Theorem}[section]
\newtheorem{lemma}[theorem]{Lemma}

\newtheorem{corollary}[theorem]{Corollary}
\newtheorem{proposition}[theorem]{Proposition}
\newtheorem{remark}[theorem]{Remark}

\begin{document}

\title{\textbf{A dynamic approach to a proximal-Newton method for monotone inclusions in Hilbert spaces, with complexity $\bigo(1/n^2)$}}

\author{H. Attouch\thanks{I3M UMR CNRS 5149,
    Universit\'e Montpellier II, Pl. E. Bataillon,
    34095 Montpellier, France
    ({\tt hedy.attouch@univ-montp2.fr})}
 \and
 M. Marques Alves \thanks{Department of Mathematics,
  Federal University of Santa Catarina,
  88.040-900  Florian\'opolis-SC, Brazil ({\tt maicon.alves@ufsc.br})}
 \and
 Benar F. Svaiter\thanks{IMPA, Estrada Dona Castorina 110,
   22460-320 Rio de Janeiro, Brazil ({\tt benar@impa.br})
   tel: 55 (21) 25295112, fax: 55 (21)25124115.  }\hspace{.5em}
 \thanks{Partially supported by CNPq
   grants 474996/2013-1, 302962/2011-5, FAPERJ grant E-26/102.940/2011
   and by PRONEX-Optimization}
}
\maketitle
\begin{abstract}
  In a Hilbert setting, we introduce a new dynamical system and associated algorithms  for solving monotone inclusions by rapid methods.
  Given  a  maximal monotone operator $A$, the evolution is governed by the time dependent operator $I -(I + \lambda(t) {A})^{-1}$, where  the positive control
parameter $\lambda(t)$	tends to infinity as $t \to + \infty$. The  tuning of $ \lambda (\cdot) $ is done in a closed-loop way, by  resolution of the algebraic equation
$\lambda \norm{(I + \lambda {A})^{-1}x -x}=\theta$, where $\theta $ is a positive given constant.
The existence and uniqueness of a  strong  global solution for the Cauchy problem follows from Cauchy-Lipschitz theorem. We prove the weak convergence of the trajectories to equilibria, and superlinear convergence under an error bound condition. When $A =\partial f$ is the  subdifferential of a closed convex  function $f$,
we show  a $\bigo(1/t^2)$ convergence property of $f(x(t))$ to the infimal value of the problem.
Then, we introduce proximal-like algorithms which can be  obtained  by time
discretization of the continuous dynamic, and which share the same fast convergence properties.  As distinctive features, we allow a relative error tolerance for the solution of the proximal
subproblem similar to the ones proposed in ~\cite{So-Sv1, So-Sv2}, and a large step condition, as proposed in~\cite{MS1,MS2}.
For general convex minimization problems, the complexity  is $\bigo(1/n^2)$. In the regular case,  we show the global quadratic convergence of an associated proximal-Newton method.

\vspace{0.5cm}


\end{abstract}

 \bigskip

 \bigskip

\section*{Introduction}
Let $ \mathcal H $ be a real Hilbert space, and ${A}: \mathcal H \tos \mathcal H$  be a maximal
monotone operator.
The space $\mathcal H$ is endowed with the scalar product
$\left\langle . , . \right\rangle $, with  \  $\|x\|^2 =
\left\langle x , x \right\rangle $  for any $x\in \mathcal H$.
Our goal is to develop new continuous and discrete dynamics, with properties of fast convergence, designed to solve the equation
\begin{equation}\label{eq:basic1}
\mbox{find} \  x\in \mathcal H \mbox{ such that }  \ 0 \in A x.
\end{equation}
 We start from the classical method, which consists in formulating \eqref{eq:basic1} as a fixed point problem:
\begin{equation}\label{eq:basic2}
\mbox{find} \  x\in \mathcal H \mbox{ such that }  \ x -  \left(I + \lambda {A}\right)^{-1}x =0,
\end{equation}
where $\lambda >0$ is a positive parameter, and $\left(I + \lambda {A}\right)^{-1}$ is the resolvent of index $\lambda$ of $A$
(recall that the resolvents are non expansive mappings from $\mathcal H$ into $\mathcal H$). Playing on the freedom of choice of the parameter $\lambda >0$, we are led to consider the  evolution problem:
\begin{equation}\label{eq:basic3}
\dot x(t) + x(t) - (I + \lambda(t) {A})^{-1}x(t)=0.
\end{equation}
When $\lambda(\cdot)$ is locally absolutely continuous, this differential equation falls within Cauchy-Lipschitz theorem.
Then, the strategy is to choose a control variable $t\mapsto \lambda (t)$ which gives good properties of asymptotic convergence of  \eqref{eq:basic3}.
In standard methods for solving monotone inclusions, the parameter $\lambda(t)$  ($\lambda_k $ in the discrete algorithmic case) is prescribed to stay bounded away from zero and infinity.
By contrast, our strategy is to let $\lambda(t)$  tend to $+ \infty$ as $t\to +\infty$. This will be a crucial ingredient for obtaining fast convergence properties.
But the precise tuning of $\lambda (\cdot)$ in such an open-loop way is a difficult task, and the open-loop approach raises numerical difficulties.
Instead, we  consider the following  system  \eqref{eq:edomm} with variables $(x, \lambda)$, where the tuning is done in a closed-loop way via the second equation of \eqref{eq:edomm}
($\theta$ is a fixed positive parameter):
\begin{equation}\label{eq:edomm}
{\rm (LSP)} \ \left\{
\begin{array}{l}
\dot x(t) + x(t) -(I + \lambda(t) {A})^{-1}x(t) =0,\qquad \lambda (t)>0,\qquad	   \\
\rule{0pt}{20pt}
\lambda(t)\norm{(I + \lambda(t) {A})^{-1}x(t)-x(t)}=\theta.\\
\end{array}\right.
\end{equation}
Note that $\lambda (\cdot)$ is an unknown function, which is obtained by solving this system. When the system is asymptotically stabilized, i.e., $\dot x(t) \to 0$, then the second equation of \eqref{eq:edomm} forces
$\lambda(t) =\frac{\theta}{\| \dot x(t) \|}$ to  tend to $+ \infty$ as $t\to +\infty$. 
Our main results can be summarized as follows:

In Theorem \ref{th:main.1}, we  show that, for any given  $x_0\in \mathcal H \setminus
{A}^{-1}(0)$, and $\theta>0$, there exists a unique strong (locally Lipschitz in time) global solution $t \mapsto (x(t), \lambda (t))$	of \eqref{eq:edomm} which satisfies the Cauchy data
$x(0)= x_0$.

 In Theorem \ref{th:main.2}, we study the asymptotic behaviour of the orbits of \eqref{eq:edomm}, as $t\rightarrow + \infty$.
Assuming  $A^{-1} (0) \neq \emptyset$, we show that	for any orbit $t
  \mapsto (\lambda(t),x(t))$ of  \eqref{eq:edomm}, $\lambda (t)$ tends increasingly to $+\infty$, and $w-\lim_{t \rightarrow +\infty} x(t) = x_{\infty}$
    exists, for some $x_{\infty} \in A^{-1} (0)$.  We complete these results by showing  in Theorem \ref{thm-strg-conv} the strong convergence of the trajectories under certain additional properties, and in Theorem \ref{thm-error-bound}  superlinear convergence under an error bound assumption.
    
 In  Theorem \ref{thm-reg-Newton}, we show 
   that  \eqref{eq:edomm} has a natural link with the regularized Newton dynamic, which was introduced in \cite{AS}.
In fact, $\lambda(t)$  tends to $+ \infty$ as $t\to +\infty$ is equivalent to  the convergence to zero of the coefficient of the regularization term (Levenberg-Marquardt type) in the regularized Newton dynamic. 
  Thus \eqref{eq:edomm}  is likely to share some of the nice convergence properties of the Newton method.
  
In Theorem   \ref{th:complexity}, when $A =\partial f$ is the  subdifferential of a convex lower semicontinuous proper function $f: \mathcal H \to \mathbb R \cup \left\{+\infty\right\}$, we show the $\bigo(1/t^2)$ convergence property
\begin{align*}
  f(x(t))-\inf_{\mathcal H} f\leq \frac{C_1}{(1+ C_2t)^2}.
\end{align*}
In Appendix~\ref{sec:ex} we consider some situations where an explicit computation of the
continuous orbits can be made, and so confirm the theoretical results.

Then, we present new algorithms which can be obtained by time
discretization of \eqref{eq:edomm}, and which share similar fast
convergence properties.  We study the iteration complexity of a
variant of the proximal point method for optimization. Its
main distinctive features are:

i) a relative error tolerance for the solution of the proximal
subproblem similar to the ones proposed in ~\cite{So-Sv1, So-Sv2}, see
also \cite{ABS} in the context of semi-algebraic and
tame optimization;

ii) a large step condition, as proposed in~\cite{MS1,MS2}. Let us notice
that the usefulness of letting the parameter $\lambda_k$ tends to
infinity in the case of the proximal algorithm, was already noticed by
Rockafellar in \cite{Rock} (in the case of a strongly monotone operator,
he showed a superlinear convergence property).

Cubic-regularized Newton method was first proposed
  in~\cite{GriewankCubic} and, after that, in~\cite{MR2319241}.
As a main result, in Theorem \ref{th:e0} we show that the complexity of our
method is $\bigo(1/n^2)$, the same as the one of the
cubic-regularized Newton method~\cite{NP}.

For  smooth convex optimization we introduce a corresponding proximal-Newton method, which has rapid global convergence properties (Theorem \ref{th:main}), and has quadratic convergence in the regular case (Theorem \ref{quad-conv-thm}).


\section{Study of the algebraic relationship linking $\lambda$ and
  \textit{x} }
Let us fix  $\theta >0$ a positive parameter.	 
We start by analyzing the algebraic relationship
 \begin{equation}
\lambda \norm{(I + \lambda {A})^{-1}x -x}=\theta,
 \end{equation}
that links  variables $\lambda \in ]0,+\infty[$ and $x\in \mathcal H$ in the second equation of  \eqref{eq:edomm}.
Define
\begin{align}
  \label{eq:psi}
  \varphi: \  [0,\infty[\times \mathcal H \to\R^+,\qquad
  \varphi(\lambda,x)=\lambda\norm{x-(I + \lambda {A})^{-1}x} \  \ \mbox{for} \ \lambda >0, \ \  \varphi(0,x)=0.
\end{align}
We denote by $J_{\lambda}^{A} = (I+ \lambda A)^{-1}$ the resolvent of
index $\lambda > 0$ of $A$, and by $ A_{\lambda} = \frac{1}{\lambda}
\left(I - J_{\lambda}^{A}\right)$ its Yosida approximation of index
$\lambda > 0$.	
To analyze the dependence of $\varphi$ with respect to
$\lambda$ and $x$, we recall some classical facts concerning resolvents
of maximal monotone operators.
\begin{proposition}
\label{resolv. cont}
  For any $\lambda >0$, $\mu >0$, and any $x\in \mathcal H$, the following properties hold:
\begin{align}\label{resolv-x-lip}
& i) \	J_{\lambda}^{A}: \mathcal H \to \mathcal H \ \mbox{is nonexpansive, and } \ A_{\lambda}: \mathcal H \to \mathcal H \  \mbox{is} \ \frac{1}{\lambda}- \mbox{Lipschitz continuous}.\\
\label{resoleq}
 & ii) \ J_{\lambda}^{A} x =   J_{\mu}^{A} \left( \frac{\mu}{\lambda} x  + \left(1 -  \frac{\mu}{\lambda} \right) J_{\lambda}^{A}x \right);\\
\label{local.lip1}
 & iii) \  \| J_{\lambda}^{A}x	- J_{\mu}^{A}x \|  \leq  |\lambda - \mu|  \  \|A_{\lambda} x \|; \\
 \label{resolv-asympt1}
 & iv) \  \lim_{\lambda \to 0}J_{\lambda}^{A}x = \mbox{{\rm proj}}_{\overline{\dom (A)}}x;\\
 \label{resolv-asympt2}
 & v) \  \lim_{\lambda \to + \infty }J_{\lambda}^{A}x = \mbox{{\rm proj}}_{A^{-1} (0)}x,     \quad  \mbox{if} \  A^{-1} (0) \neq \emptyset.
\end{align}
\noindent As a consequence, for any $x\in \mathcal H$ and any $0< \delta < \Lambda <+\infty$, the function $\lambda \mapsto  J_{\lambda}^{A}x$ is Lipschitz continuous on  $\left[\delta, \Lambda  \right]$.
 More precisely, for any
 $\lambda, \mu $ belonging to $\left[\delta, \Lambda  \right]$
\begin{equation}
\label{local.lip2}
\| J_{\lambda}^{A}x  - J_{\mu}^{A}x \| \leq  |\lambda - \mu| \	\|A_{\delta} x \|.
\end{equation}
\end{proposition}

\begin{proof} $i)$ is a classical result, see \cite[Proposition 2.2, 2.6]{Br}.\\
$ii)$ \ Equality (\ref{resoleq}) is known as the resolvent equation, see \cite{Br}. Its proof is  straightforward:
 By definition of $\xi = J_{\lambda}^{A}x$, we have
\begin{center}
 $\xi + \lambda A\xi \ni x$,
\end{center}
which, after multiplication by $\frac{\mu}{\lambda}$, gives
\begin{center}
 $\frac{\mu}{\lambda} \xi + \mu  A\xi \ni \frac{\mu}{\lambda} x$.
\end{center}
By adding $\xi$ to the two members of the above equality, we obtain
\begin{center}
 $\xi + \mu  A\xi \ni \frac{\mu}{\lambda} x - \frac{\mu}{\lambda} \xi +\xi$,
\end{center}
which gives the desired equality
\begin{center}
 $\xi	 = J_{\mu}^{A} \left( \frac{\mu}{\lambda} x  + \left(1 -  \frac{\mu}{\lambda} \right) J_{\lambda}^{A}x \right)$.
\end{center}
$iii)$ \ For any $\lambda >0$, $\mu >0$, and any $x\in \mathcal H$, by using successively the resolvent equation and the nonexpansive property of the resolvents, we have
\begin{align*}
\| J_{\lambda}^{A}x  - J_{\mu}^{A}x \| & = \| J_{\mu}^{A} \left( \frac{\mu}{\lambda} x	+ \left(1 -  \frac{\mu}{\lambda} \right) J_{\lambda}^{A}x \right)  - J_{\mu}^{A}x \| \\
& \leq	\| \left(1 -  \frac{\mu}{\lambda} \right)  \left(x - J_{\lambda}^{A}x \right) \| \\
& \leq |\lambda - \mu| \ \|A_{\lambda} x \|.
\end{align*}
\noindent Using that $\lambda \mapsto \|A_{\lambda} x \|$ is nonincreasing, (see \cite[Proposition 2.6]{Br}), we  obtain (\ref{local.lip2}).\\
$iv)$ see \cite[Theorem 2.2]{Br}.\\
$v)$ \ It is the viscosity selection property of the Tikhonov approximation, see \cite{Att-visc}.
\end{proof}
Let us first consider the mapping $x \mapsto \varphi (\lambda, x)$. Noticing that, for $\lambda>0$, $\varphi(\lambda,x)= {\lambda}^2 \norm{A_{\lambda}x}$,
the following result is just the reformulation in terms of $\varphi$ of the $\frac{1}{\lambda}$-Lipschitz continuity of $A_{\lambda}$.
\begin{proposition}
  \label{pr:psi0}
  For any $x_1,x_2\in {\mathcal H}$ and $\lambda>0$,
    \begin{align*}
      \Abs{\varphi(\lambda,x_1)-\varphi(\lambda,x_2)}\leq\lambda\norm{x_2-x_1}.
  \end{align*}
\end{proposition}

The next result was proved in~\cite[Lemma 4.3]{MS2} for finite
dimensional spaces. Its proof for arbitrary Hilbert spaces is similar and
is provided for the sake of completeness.

\begin{lemma}
  \label{lm:psi1}
  For any $x\in {\mathcal H}$ and $0<{\lambda_1}\leq{\lambda_2}$,
  \begin{align}
    \label{eq:psi1}
    \dfrac{{\lambda_2}}{{\lambda_1}}\,\varphi({\lambda_1},x)\leq\varphi({\lambda_2},x)\leq
    \left(\dfrac{{\lambda_2}}{{\lambda_1}}\right)^{2}\varphi({\lambda_1},x)
 \end{align}
 and $\varphi({\lambda_1},x)=0$ if and only if $0\in {A}(x)$.
\end{lemma}

\begin{proof}
  Let $y_i=J_{\lambda_i}^Ax$ and $v_i=A_{\lambda_i}x$ for $i=1,2$.
  In view of these definitions,
  \begin{align*}
    v_i\in A(y_i),\qquad \lambda_iv_i+y_i-x=0\qquad i=1,2.
  \end{align*}
  Therefore,
  \begin{align*}
    {\lambda_1}({v_1}-{v_2})+{y_1}-{y_2}=({\lambda_2}-{\lambda_1}) {v_2},
    \qquad
    {v_2}-{v_1}+{\lambda_2}^{-1}({y_2}-{y_1})
    =({\lambda_1}^{-1}-{\lambda_2}^{-1})({y_1}-x).
 \end{align*}
  Since ${A}$ is monotone, the inner products of both sides of the first
  equation by ${v_1}-{v_2}$ and of the second equation by ${y_2}-{y_1}$ are
  non-negative. Since ${\lambda_1}\leq{\lambda_2}$,
  \begin{equation*}
    \inner{{v_1}-{v_2}}{{v_2}}\geq 0,\quad \inner{{y_2}-{y_1}}{{y_1}-x}\geq 0,
    \quad  \norm{{v_1}}\geq\norm{{v_2}},\quad \norm{{y_2}-x}\geq\norm{{y_1}-x}.
  \end{equation*}
  The two inequalities in \eqref{eq:psi1} follow from the two last inequalities
  in the above equation and definition \eqref{eq:psi}.
  The last part of the proposition follows trivially from the maximal
  monotonicity of ${A}$ and definition \eqref{eq:psi}.
\end{proof}

We can now analyze the properties of the mapping $\lambda \mapsto
\varphi (\lambda, x)$. Without ambiguity, we write shortly $J_{\lambda}$
for the resolvent  of
index $\lambda > 0$ of $A$.
\begin{proposition}
  \label{pr:phi-property}
  For any $x\notin A^{-1} (0)$,  the function
   $\lambda \in [0,\infty[ \  \mapsto \varphi (\lambda, x) \in \R^+$ is continuous, strictly increasing, $\varphi(0,x)=0$, and $\lim_{\lambda \to + \infty}\varphi (\lambda, x) =+ \infty$.
\end{proposition}
\begin{proof} 
 It follows from \eqref{eq:psi}
    and the first inequality in \eqref{eq:psi1} with $\lambda_2=1$, 
    $\lambda=\lambda_1\leq1$ that
  \[
  0\leq \lim\sup_{\lambda\to0^+}\varphi(\lambda,x)\leq\lim_{\lambda\to0^+}
  \lambda\varphi(1,x)=0,
  \]
  which proves continuity of $\lambda \mapsto \varphi(\lambda,x)$ at $\lambda=0$.
  Note that this also results from Proposition \ref{resolv. cont} $iv)$.
  Since $0\notin A(x)$, it follows from
  the last statement in Lemma~\ref{lm:psi1} and the first inequality in
  \eqref{eq:psi1} that $\lambda\mapsto\varphi(\lambda,x)$ is strictly
  increasing, and that
  $\lim_{\lambda\to\infty}\varphi(\lambda,x)=+\infty$. Left-continuity
  and right-continuity of $\lambda\mapsto\varphi(\lambda,x)$ follows
  from the first and the second inequality in~\eqref{eq:psi1}.
\end{proof}
In view of Proposition \ref{pr:phi-property}, if $0\notin {A}(x)$ there exists a
unique $\lambda>0$ such that $\varphi(\lambda,x)=\theta$.
%
%
It remains to analyze how such a $\lambda$ depends on $x$.
Define, for $\theta>0$
\begin{align}
  \label{eq:lambda}
  \begin{aligned}
  &\Omega={\mathcal H}\setminus {A}^{-1}(0),\\
  &\Lambda_\theta:\Omega\to ]0,\infty[,\quad \Lambda_\theta(x)=
  \left(\varphi(\cdot,x)\right)^{-1}(\theta).
  \end{aligned}
\end{align}
Observe that $\Omega$ is open. More precisely, 
\begin{align}
  \label{eq:wo}
  \left\{z\in {\mathcal H}\;\left|\; \norm{z-x} <
      \dfrac{\theta}{\Lambda_\theta(x)}\right\}\right.&\subset
  \Omega,\qquad\forall x\in \Omega.
\end{align}
To prove this inclusion, suppose that
$\norm{z-x}<\theta/\Lambda_\theta(x)$. By the triangle inequality and
Proposition~\ref{pr:psi0} we have
\begin{align*}
  \varphi(\Lambda_\theta(x),z) \geq 
  \varphi(\Lambda_\theta(x),x) -\Abs{\varphi(\Lambda_\theta(x),z)
  -\varphi(\Lambda_\theta(x),x)}
  \geq \theta-\Lambda_\theta(x)\norm{z-x}>0.
\end{align*}
Hence, $z\notin A^{-1}(0)$.

Function $\Lambda_\theta$
allows us to express  \eqref{eq:edomm} as an autonomous  EDO:
\begin{equation}\label{eq:edoa}
\left\{
\begin{array}{l}
\dot x(t) + x(t) - \left(I + \Lambda_\theta(x(t)) {A}\right)^{-1}x(t) = 0;    \\
\rule{0pt}{20pt}
x(0)=x_0.
\end{array}\right.
\end{equation}
In order to study the properties of the function $\Lambda_\theta$, it is convenient to define
\begin{align}
  \label{eq:g}
  \Gamma_\theta(x)=\min\{\alpha>0\;|\; \norm{x-(I + \alpha^{-1}{A})^{-1}x}\leq
  \alpha\theta\}.
\end{align}
\begin{lemma}
  \label{lm:g}
  The function $\Gamma_\theta :   \mathcal H \rightarrow \mathbb R^+$ is Lipschitz continuous with constant $1/\theta$ and
  \begin{align*}
    \Gamma_\theta(x)&=
    \begin{cases}
      1/\Lambda_\theta(x),& if \ x\in \Omega\\
     0,&\text{otherwise}
    \end{cases}
  \end{align*}
\end{lemma}

\begin{proof}
The first inequality in \eqref{eq:psi1} is equivalent to saying that
$\lambda\mapsto\norm{x-(I + \lambda {A})^{-1}x}$ is
 a non-decreasing function.  Therefore, $\alpha\mapsto\norm{x-(I + \alpha^{-1}
    {A})^{-1}x}$ is a (continuous) non-increasing function. As a consequence, the set
  \begin{align*}
    \{\alpha>0\;|\; \norm{x-(I + \alpha^{-1}{A})^{-1}x}\leq
  \alpha\theta\}
  \end{align*}
  is always a nonempty interval, and $\Gamma_\theta$ is a real-valued non-negative
  function. The relationship between $\Gamma_\theta(x)$ and $\Lambda_\theta(x)$ is straightforward:
  by definition, if \ $x\in \Omega$
 \begin{align*}
\Gamma_\theta(x)&=\min\{\alpha>0\;|\; \frac{1}{\alpha}\norm{x-(I + \frac{1}{\alpha}{A})^{-1}x}\leq \theta\},\\
&= \frac{1}{\sup\{\lambda \;|\; \lambda\norm{x-(I + \lambda{A})^{-1}x}\leq \theta\}   },   \\
& = 1/\Lambda_\theta(x).
  \end{align*}
  Moreover, if $x \in S$, then for any $\alpha>0$, $x-(I + \alpha^{-1}{A})^{-1}x=0$, and $\Gamma_\theta(x) =0$.\\
 Let us now show that  $\Gamma_\theta$ is Lipschitz continuous.
Take $x,y\in {\mathcal H}$ and $\alpha>0$. Suppose that $ \norm{x-(I +\alpha^{-1}{A})^{-1}x}\leq
  \alpha\theta$.  We use that $x \mapsto \norm{x-(I + \lambda{A})^{-1}x}$ is nonexpansive (a consequence of the equality $\norm{x-(I + \lambda{A})^{-1}x}= \norm{\lambda A_{\lambda}x}$ and  Proposition \ref{resolv. cont}, item $i)$). Hence
  \begin{align*}
     \norm{y-(I +\alpha^{-1}{A})^{-1}y} & \leq \norm{x-(I + \alpha^{-1}{A})^{-1}x}+
     \norm{y-x}\\
     & \leq
  \alpha\theta+\norm{y-x}\\
  &=\left(\alpha+\dfrac{\norm{y-x}}{\theta}\right)\theta.
  \end{align*}
  Let $\beta=\alpha+\norm{y-x}/\theta$. Since $\beta\geq \alpha$, by using again that $\lambda\mapsto\norm{x-(I + \lambda^{-1}
    {A})^{-1}x}$ is a  non-increasing function,
  \begin{align*}
    \norm{y-(I + \beta^{-1}{A})^{-1}y}\leq  \norm{y-(I + \alpha^{-1}{A})^{-1}y}
   \leq \beta\theta.
  \end{align*}
  By definition of $\Gamma_\theta$, we deduce that $\Gamma_\theta(y)\leq
  \beta= \alpha+\norm{y-x}/\theta$. This being true for any $\alpha \geq
  \Gamma_\theta(x)$, it follows that
  $\Gamma_\theta(y)\leq\Gamma_\theta(x)+\norm{y-x}/\theta$.  Since the
  same inequality holds by interchanging $x$ with $y$, we conclude that
  $\Gamma_\theta$ is $1/\theta$-Lipschitz continuous.
\end{proof}

Observe that in \eqref{eq:edomm}
\[
\lambda(t)=\Lambda_\theta(x(t)),\quad \dot x (t) = J_{\Lambda_\theta(x(t))}x(t)-x(t).
\]
We are led to study the vector field $F$ governing this EDO,
\begin{align}
  \label{eq:vff}
  F:\Omega\to{\mathcal H},\qquad F(x)=J_{\Lambda_\theta(x)}x-x.
\end{align}
\begin{proposition}
  \label{pr:lp}
  The vector field $F$ 
  is locally Lipschitz continuous.
\end{proposition}

\begin{proof}
  Take $x_0\in\Omega$ and $0<r<\theta/\Lambda_\theta(x_0)$. Set $\lambda_0  =\Lambda_\theta(x_0)$. By \eqref{eq:wo} we have
   $B(x_0,r)\subset \Omega$. In view of the choice of $r$ and
  Lemma~\ref{lm:g}, for any $x\in B(x_0,r)$
  \begin{align}
    \label{eq:baux}
  0<\dfrac{1}{\lambda_0}-\dfrac{r}{\theta}\leq
  \dfrac{1}{\Lambda_\theta(x)}=\Gamma_\theta(x) \leq
  \dfrac{1}{\lambda_0}+\dfrac{r}{\theta}.
  \end{align}
  Take $x,y\in B(x_0,r)$ and let
  \[
  \lambda=\Lambda_\theta(x),\qquad \mu=\Lambda_\theta(y).
  \]
By using that $x \mapsto \norm{x-J_\lambda(x)}$ is nonexpansive, and the resolvent equation (Proposition  \ref{resolv. cont}, item $iii)$), we have
   \begin{align*}
     \norm{F(x)-F(y)} &= \norm{J_\lambda x-x-\left(J_\mu y-y\right)} \\
     &\leq \norm{J_\lambda x-x-\left(J_\lambda y-y\right)}+
     \norm{J_\lambda y-J_\mu y}\\
     &\leq \norm{x-y}+\abs{\lambda-\mu}\dfrac{\norm{J_\mu y-y}}{\mu}\\
     &= \norm{x-y}+\dfrac{\abs{\lambda-\mu}}{\mu^2}\;\theta
   \end{align*}
   where the last equality follows from the definition of $\mu$
   and \eqref{eq:lambda}. Using  Lemma~\ref{lm:g} we have
  \begin{align*}
    \dfrac{\abs{\lambda-\mu}}{\mu^2} &=
    \dfrac{\lambda}{\mu}\Abs{\dfrac{1}{\mu}-\dfrac{1}{\lambda}}\\
    &= \dfrac{\Gamma_\theta(y)}{\Gamma_\theta(x)}
    \abs{\Gamma_\theta(y)-\Gamma_\theta(x)} \leq \frac{1}{\theta}
    \dfrac{\Gamma_\theta(y)}{\Gamma_\theta(x)}\norm{y-x}.
  \end{align*}
  In view of \eqref{eq:baux},
  \begin{align*}
    \dfrac{\Gamma_\theta(y)}{\Gamma_\theta(x)} \leq
    \dfrac{\theta+\lambda_0r}{\theta-\lambda_0r}.
  \end{align*}
   Combining the three above results, we conclude that
   \begin{align*}
     \norm{F(x)-F(y)}\leq \left[
       1+\dfrac{\theta+\lambda_0r}{\theta-\lambda_0r} \right]
     \norm{x-y}=\dfrac{2\theta}{\theta-\lambda_0r}\norm{x-y}.
   \end{align*}
   which is the desired result.
\end{proof}

\section{Existence and uniqueness of a global solution}
Given $x_0\in\Omega = \mathcal H \setminus A^{-1}(0)$, we study the Cauchy problem
\begin{equation}\label{eq:edomm-cauchy}
\left\{
\begin{array}{l}
\dot x(t) + x(t) -(I + \lambda(t) {A})^{-1}x(t) =0,\qquad \lambda (t)>0,    \\
\rule{0pt}{20pt}
\lambda(t)\norm{(I + \lambda(t) {A})^{-1}x(t)-x(t)}=\theta,\\
\rule{0pt}{20pt}
x(0)=x_0.
\end{array}\right.
\end{equation}
Note that the assumption $x_0\in\Omega = \mathcal H \setminus A^{-1}(0)$ is
not restrictive, since when $x_0 \in A^{-1}(0)$, the problem is already
solved. Following the results of the previous section, \eqref{eq:edomm-cauchy}
can be equivalently formulated as an autonomous EDO, with respect to the
unknown function $x$.
\begin{equation}\label{eq:Cauchy1}
\left\{
\begin{array}{l}
\dot x(t) + x(t) - \left(I + \Lambda_\theta(x(t)) {A}\right)^{-1}x(t) = 0;    \\
\rule{0pt}{20pt}
x(0)=x_0.
\end{array}\right.
\end{equation}
Let us first state a local existence result.
\begin{proposition}
  \label{pr:local}
  For any $x_0\in\Omega = \mathcal H \setminus A^{-1}(0)$, there exists some $\varepsilon>0$ such that
  \eqref{eq:edomm-cauchy}  has a unique solution $(x,\lambda):[0,\varepsilon]\to {\mathcal H}\times \R_{++}$. Equivalently,
  \eqref{eq:Cauchy1} has a unique solution $x:[0,\varepsilon ]\to {\mathcal H}$.  For this
  solution, $x(\cdot)$ is $\mathscr{C}^1$, and $\lambda(\cdot)$ is locally
   Lipschitz continuous.
\end{proposition}
\begin{proof} We use the reformulation of   \eqref{eq:edomm-cauchy}  as an autonomous differential equation, as described in  \eqref{eq:Cauchy1}. Equivalently
\[
 \dot x (t) =  F(x(t)),
\]
with $F(x)$ as in~\eqref{eq:vff}.
By Proposition~\ref{pr:lp}, the
vector field $F$ is locally Lipschitz continuous on the open set $\Omega
\subset \mathcal H$.  Hence, by Cauchy-Lipschitz theorem (local
version), for any $x_0\in\Omega$, there exists a unique local solution
$x:[0,\varepsilon ]\to {\mathcal H}$ of \eqref{eq:edoa}, for some
$\varepsilon >0$.  Equivalently, there exists a unique local solution
$(x,\lambda)$ of \eqref{eq:edomm}. Clearly $x$ is a classical
$\mathscr{C}^1$ orbit, and $t \mapsto\lambda (t) = \Lambda_\theta(x(t))=
\frac{1}{ \Gamma_\theta(x(t))}$ is Lipschitz continuous (by taking
$\epsilon$ sufficiently small), a consequence of Lemma~\ref{lm:g}, and
$x(t) \in \Omega$.
\end{proof}
In order to pass from a local to a global solution, we first establish some further properties of the map $t \mapsto \lambda(t)$.
\begin{lemma}
  \label{lm:ldot}
  If $(x, \lambda):[0,\varepsilon]\to {\mathcal H}\times \R_{++}$ is a solution
  of  \eqref{eq:edomm-cauchy}, then $|\dot\lambda (t)|\leq\lambda (t)$ for almost all
  $t\in[0,\varepsilon]$.
\end{lemma}

\begin{proof}
  Take $t,t'\in[0,\varepsilon]$, $t\neq t'$. Then
  \begin{align}
    \Abs{\lambda(t')-\lambda(t)}&=
    \lambda(t)\lambda(t')\Abs{\frac{1}{\lambda(t)}-\frac{1}{\lambda(t')}}\\
    &= \lambda(t)\lambda(t')
    \Abs{\Gamma_\theta(x(t))-\Gamma_\theta(x(t'))}\\
    &\leq \frac{\lambda(t)\lambda(t')\norm{x(t)-x'(t)}}{\theta},
  \end{align}
  where the last inequality follows from  Lemma~\ref{lm:g}. Therefore
  \begin{align}
    \lim\sup_{t'\to t}\Abs{\frac{\lambda(t')-\lambda(t)}{t'-t}}
    &\leq \lim_{t'\to t}\frac{\lambda(t)\lambda(t')\norm{x(t')-x(t)}}{\theta
      \abs{t'-t} }=
     \lambda(t)^2\norm{\dot x(t)}/\theta=\lambda(t).
  \end{align}
\end{proof}
\begin{lemma}
  \label{lm:lnond}
  If $(x, \lambda):[0,\varepsilon]\to {\mathcal H}\times \R_{++}$ is a solution of
  \eqref{eq:edomm-cauchy}, then $\lambda(\cdot)$ is non-decreasing.
\end{lemma}

\begin{proof}
  Since $\lambda$ is locally Lipschitz continuous, to prove that it is
  non-decreasing it suffices to show that $\dot\lambda(t)\geq 0$ for
  almost all $t\in[0,\varepsilon]$. Take $t\in[0,\varepsilon [$ and define
  \begin{align*}
    \mu=\lambda(t),\quad y=J_\mu x(t),\quad v=\mu^{-1}(x(t)-y).
  \end{align*}
  Observe that $v\in {A}(y)$ and $\mu v+y-x(t)=0$. Define
  \begin{align*}
    z_h=x(t)+h\dot x(t),\qquad 0<h<\min\{\varepsilon-t,1\}.
  \end{align*}
  Since $\dot x(t)=-\mu v$,
  we have $(1-h)\mu v+y-z_h=0$, $J_{(1-h)\mu}z_h=y$ and so
  \begin{align*}
    \varphi((1-h)\mu, z_h)=(1-h)\mu\norm{y-z_h} = (1-h)^2\mu\norm{y-x(t)}
    = (1-h)^2\theta.
  \end{align*}
  Therefore, using triangle inequality, the second inequality in
  Lemma~\ref{lm:psi1} and Proposition~\ref{pr:psi0}, we obtain
  \begin{align*}
    \varphi(\mu,x(t+h)) &\leq
    \varphi(\mu,z_h)+\Abs{\varphi(\mu,x(t+h))-\varphi(\mu,z_h)}\\
    &\leq  \dfrac{\varphi((1-h)\mu,z_h)}{(1-h)^2} +
     \mu\norm{x(t+h)-z_h}\\
    &= \theta+\mu\norm{x(t+h)-x(t)-h\dot x(t)}.
  \end{align*}
  To simplify the notation, define
  \begin{align*}
    \rho_h=\dfrac{\mu\norm{x(t+h)-x(t)-h\dot x(t)}}{\theta}.
  \end{align*}
  Observe that $\rho_h\geq 0$ (for $0<h<\min\{\varepsilon-t,1\}$), and
  $\lim_{h\to 0^+}\rho_h/h=0$. Now,
  the above inequality can be written as
  \begin{align*}
    \varphi(\mu,x(t+h))\leq \theta(1+\rho_h).
  \end{align*}
  It follows from this inequality, the non-negativity of $\rho_h$ and
  Lemma~\ref{lm:psi1} that
  \begin{align*}
    \varphi\left(\dfrac{\mu}{1+\rho_h},x(t+h)\right)\leq \theta.
  \end{align*}
  Since $\varphi(\cdot,x(t+h))$ is strictly increasing, and
  $\varphi(\lambda(t+h),x(t+h))=\theta$,
  \begin{align*}
    \lambda(t+h) \geq \dfrac{\mu}{1+\rho_h}=\dfrac{\lambda(t)}{1+\rho_h}.
  \end{align*}
  Therefore
  \begin{align*}
    \liminf_{h\to 0^+}\dfrac{\lambda(t+h)-\lambda(t)}{h} &\geq
    \lim_{h\to 0^+}\dfrac{1}{h}\left[\dfrac{\lambda(t)}{1+\rho_h}-\lambda(t)\right]
    = -\lim_{h\to 0^+}\lambda(t)\dfrac{\rho_h/h}{1+\rho_h}=0.
  \end{align*}
\end{proof}
In view of Proposition \ref{pr:local}, there exists a solution of \eqref{eq:Cauchy1}
defined on a maximal interval.
Next we will prove that this maximal interval
is $[0,+\infty[$.

\begin{theorem}
  \label{th:main.1} For any $x_0\in\Omega = \mathcal H \setminus A^{-1}(0)$, there exists
  a unique global solution $(x,\lambda):[0,+ \infty[\to {\mathcal H}\times \R_{++}$ of
  the Cauchy problem \eqref{eq:edomm-cauchy}. Equivalently,
  \eqref{eq:Cauchy1} has a unique solution $x:[0,+ \infty[\to {\mathcal H}$.  For this
  solution, $x(\cdot)$ is $\mathscr{C}^1$, and $\lambda(\cdot)$ is locally
   Lipschitz continuous. Moreover,

i) $\lambda(\cdot)$ is non-decreasing;

ii) $t \mapsto \norm{J_{\lambda (t)} x(t)-x(t)}$ is non-increasing;

iii) For any $0\leq t_0\leq t_1$
  \begin{align*}
    & \lambda(t_0)\leq\lambda(t_1)\leq e^{(t_1-t_0)}\lambda(t_0)\\
    & \norm{J_{\lambda(t_0)}x(t_0)-x(t_0)}e^{-(t_1-t_0)}\leq
     \norm{J_{\lambda(t_1)}x(t_1)-x(t_1)}\leq	\norm{J_{\lambda(t_0)}x(t_0)-x(t_0)}.
  \end{align*}
\end{theorem}
\begin{proof} According to a standard argument, we argue by contradiction and assume that the maximum solution $x(\cdot)$ of \eqref{eq:Cauchy1} is defined on an interval $[0, T_{max}[$ with
$T_{max} < + \infty$. By Lemmas \ref{lm:ldot} and \ref{lm:lnond}, $\lambda(\cdot)$ is non-decreasing, and satisfies $0 \leq \dot\lambda (t)\leq\lambda (t)$ for almost all
  $t\in[0, T_{max}[$. By integration of this inequation, we obtain, for any  $t\in[0, T_{max}[$
  \begin{equation}\label{eq:Cauchy2}
 0 < \lambda (0) \leq  \lambda (t) \leq \lambda (0) e^{t}.
\end{equation}
Since  $t\leq  T_{max}$, we infer that $\lim_{t \to T_{max} }  \lambda (t):= \lambda_m$ exists and is finite. Moreover, by \eqref{eq:edomm-cauchy}
  \begin{equation}\label{eq:Cauchy3}
  \| \dot{x} (t)\| = \norm{(I + \lambda(t) {A})^{-1}x(t)-x(t)}= \frac{\theta}{\lambda(t)}.
 \end{equation}
 Combining \eqref{eq:Cauchy2} and \eqref{eq:Cauchy3}, we obtain that $\| \dot{x} (t)\|$ stays bounded when $t\in[0, T_{max}[$. By a classical argument, this implies that
 $\lim_{t \to T_{max} }  x (t):= x_m$ exists.

Moreover, by the second inequality in \eqref{eq:Cauchy2}, $\norm{(I + \lambda(t) {A})^{-1}x(t)-x(t)}= \frac{\theta}{\lambda(t)}$ stays bounded away from zero. Hence, at the limit, we have $\norm{(I + \lambda_m {A})^{-1}x_m-x_m} \neq 0$, which
means that $x_m\in\Omega = \mathcal H \setminus A^{-1}(0)$. Thus, we can apply again the local existence result, Proposition \ref{pr:local}, with Cauchy data $x_m$, and so obtain a solution defined on an interval strictly larger than $[0, T_{max}[$. This is a clear contradiction. Properties $i), ii), iii)$ are direct consequence of  Lemmas \ref{lm:ldot} and \ref{lm:lnond}. More precisely,
by integration of $0 \leq \dot\lambda (t)\leq\lambda (t)$ between $ t_0$ and $ t_1 \geq t_0$, we obtain $\lambda(t_0)\leq\lambda(t_1)\leq e^{(t_1-t_0)}\lambda(t_0)$. As a consequence
\begin{equation*}
 \norm{J_{\lambda(t_1)}x(t_1)-x(t_1)} = \frac{\theta}{\lambda(t_1)} \leq \frac{\theta}{\lambda(t_0)} = \norm{J_{\lambda(t_0)}x(t_0)-x(t_0)},
\end{equation*}
and
\begin{equation*}
\norm{J_{\lambda(t_1)}x(t_1)-x(t_1)} = \frac{\theta}{\lambda(t_1)} = \frac{\theta}{\lambda(t_0)} \times \frac{\lambda(t_0)}{\lambda(t_1)}\geq \norm{J_{\lambda(t_0)}x(t_0)-x(t_0)}e^{-(t_1-t_0)}.
\end{equation*}

\begin{remark} {\rm Property $iii)$ of Theorem \ref{th:main.1}, with $t_0 =0$, namely
$\norm{J_{\lambda(0)}x_0-x_0}e^{-t}\leq \norm{J_{\lambda(t)}x(t)-x(t)}$, implies that for all $t\geq 0$, we have $J_{\lambda(t)}x(t)-x(t) \neq 0$. Equivalently $x(t) \notin A^{-1}(0)	$, i.e., the system cannot be stabilized in a finite time.
Stabilization can be achieved only asymptotically, which is the subject of the next section.}
\end{remark}

\end{proof}

\section{Asymptotic behavior}

\subsection{Weak convergence} To prove the weak convergence of trajectories
of system \eqref{eq:edomm},
we  use the classical Opial lemma \cite{Op},	that we recall in its continuous form;
see also \cite{Bruck},	who initiated the use of this argument to analyze  the asymptotic convergence of nonlinear contraction semigroups in Hilbert spaces.

\begin{lemma}\label{Opial} Let	 $S$ be
a non empty subset of $\mathcal H$, and	$x: [0, +\infty [ \to \mathcal H$ a map. Assume that
\begin{eqnarray*}
(i) &&\mbox{for every }z\in S, \   \lim_{t \to	+\infty}	 \| x(t)- z	    \|	\mbox{ exists};\\
(ii) &&\mbox{every weak sequential cluster point of the map } x \mbox{ belongs to }S.
\end{eqnarray*}
Then
 $$
w-\lim_{t \to  +\infty} x(t) = x_{\infty}  \   \   \mbox{ exists,  for some element}x_{\infty}\in S.
$$
\end{lemma}

Let us state our main convergence result.

\begin{theorem}
  \label{th:main.2}
  Suppose that $A^{-1} (0) \neq \emptyset$. Given $x_0\notin A^{-1} (0)$, let
 $(x,\lambda):[0,+ \infty[\to {\mathcal H}\times \R_{++}$  be the unique global solution   of
  the Cauchy problem \eqref{eq:edomm-cauchy}. Set $d_0 = d(x_0,  A^{-1} (0))  $ the distance from $x_0$ to $A^{-1} (0)$.
  Then, the following
  properties hold:

\smallskip

i) $\norm{\dot x(t)}= \norm{x(t)- J_{\lambda (t)}x(t)}\leq
    d_0/\sqrt{2t}$; \ hence \   $ \lim_{t \rightarrow +\infty}  \norm{\dot x(t)}=0$;
    
\smallskip
    
ii) $ \lambda(t) \geq \theta\sqrt{2t}/d_0$; \ hence \   $ \lim_{t \rightarrow +\infty} \lambda(t)= + \infty$;

\smallskip

iii) $w-\lim_{t \rightarrow +\infty} x(t) = x_{\infty}$
    exists, for some $x_{\infty} \in A^{-1} (0)$.
    
\smallskip

\noindent Moreover, for any $z \in A^{-1} (0)$,  $ \|x(t) - z \|$ is decreasing.
\end{theorem}
\begin{proof}
  Define
  \begin{equation}
    \label{eq:dv}
    v(t)=\lambda(t)^{-1}(x(t)-J_{\lambda(t)}x(t)).
  \end{equation}
    Observe that $v(t)\in
  A(J_{\lambda (t)}x(t))$ and $\lambda (t) v(t)+J_{\lambda (t)}x(t) -x(t)=0$.
  For any $z \in A^{-1} (0)$, and any $t\geq 0$ set
  \begin{equation}
    \label{eq:dhz}
    h_z (t) :=	\frac{1}{2} \| x(t) - z\|^2.
  \end{equation}
  After derivation of $h_z $, and using the differential relation in \eqref{eq:edomm-cauchy} we obtain
  \begin{align}\label{basic-rate}
    \dot{h}_z (t) &=  \left\langle  x(t) - z,  \dot{x} (t)   \right\rangle\\
    & = - \left\langle x(t) - z, x(t)- J_{\lambda (t)}x(t) \right\rangle=
    -\norm{x(t) - J_{\lambda (t)}x(t)}^2-\inner{J_{\lambda (t)}x(t) -z}{\lambda (t)v(t)} \label{strong}.
  \end{align}
  Since $v(t)\in A(J_{\lambda (t)}x(t))$, $0\in A(z)$, and $A$ is (maximal) monotone
  \begin{equation}
    \label{eq:hd}
    \dot h_z(t)\leq-\norm{x(t)-J_{\lambda (t)}x(t)}^2 .
  \end{equation}
  Hence, $h_z$ is non-increasing. Moreover, by integration of \eqref{eq:hd},   for any $t>0$
  \begin{align*}
    \frac{1}{2}\norm{z-x(0)}^2\geq h_z(0)-h_z(t)&=-\int_0^t\dot h_z(u)du \\
    &\geq\int_0^t\norm{J_{\lambda(u)}x(u)-x(u)}^2du\geq t\norm{J_{\lambda(t)}x(t)-x(t)}^2
  \end{align*}
  where the last inequality follows from
  $t \mapsto \norm{J_{\lambda(t)}x(t)-x(t)}$ being non-increasing (see
  Theorem \ref{th:main.1}, $ii)$). Item $i)$ follows trivially from
  the above inequality. Item $ii)$ follows from item $i)$ and
  the algebraic relation between $x$ and $\lambda$ in
  \eqref{eq:edomm-cauchy}. To prove item $iii)$, we use
  Lemma~\ref{Opial} with $S=A^{-1}(0)$. Since $z$ in \eqref{eq:dhz} is a
  generic element of $A^{-1}(0)$, it follows from \eqref{eq:hd} that
  item {\it (i)} of Lemma~\ref{Opial} holds. Let us now prove that item
  {\it (ii)} of Lemma~\ref{Opial} also holds. Let $x_{\infty} $ be a
  weak sequential cluster point of the orbit $x(\cdot)$. Since
  $\norm{x(t) -J_{\lambda(t)}x(t)}\to 0$ as $t\to\infty$, we also have
  that $x_{\infty} $ is a weak sequential cluster point of
  $J_{\lambda (\cdot)} x(\cdot)$. Now observe that in view of items $i)$
   and $ii)$, for any $t>0$
  \begin{equation}
    \label{eq:to.be.used}
    \norm{v(t)}\leq\frac{d_0^2}{2\theta t}.
  \end{equation}
  Hence, $v(t)$ converges strongly to zero as $t$ tends to infinity.
  Since $v(t)\in A(J_{\lambda(t)}x(t))$, and the graph of $A$ is
  demi-closed, we obtain $0 \in A (x_{\infty})$, i.e.,
  $x_{\infty} \in S$.
\end{proof}

\subsection{Superlinear convergence under an error bound assumption}

In this section, we assume that the solution set $S=A^{-1}(0)$ is non-empty and that,
whenever $v\in A(x)$ is ``small'', its norm provides a
bound for the distance of $x$ to $S$. Precisely,

\smallskip

\textbf{A0)} \ $S=A^{-1}(0)$ is non-empty, and there exists $\varepsilon,\kappa>0$ such that
\begin{align*}
  v\in A(x),\;\norm{v}\leq \varepsilon\Longrightarrow  
  d(x,S) \leq \kappa \norm{v}.
\end{align*}
\begin{theorem} \label{thm-error-bound}
 Assuming {\rm \textbf{A0)}}, then $x(t)$ converges strongly to some  $x^*\in A^{-1}(0)$, and for any $\alpha\in(0,1)$ there exist positive reals
  $c_0,c_1,c_2,c_3$ such that  
  \begin{align*}
   d(x(t),S)\leq c_0 e^{-\alpha t},\;
    \lambda(t)\geq c_1 e^{\alpha t},\;
    \norm{v (t)}\leq c_2 e^{-2\alpha t},\;
    \norm{x (t)-x^*}\leq c_3 e^{-\alpha t}.
  \end{align*}
\end{theorem}
\begin{proof} 
Let $P_S$ be the projection on the closed convex set $S=A^{-1}(0)$.
  Define, for $t\geq 0$,
  \begin{align*}
    x^*(t)=P_S(x(t)),\qquad y^*(t)=P_S(y(t)).
  \end{align*}
  It follows from the assumption $A^{-1}(0)\neq\emptyset$, and from 
  \eqref{eq:to.be.used} (inside the proof of Theorem~\ref{th:main.2}) that
  $\lim_{t\to\infty}v(t)=0$.  By \textbf{A0)}, and $v(t)=\lambda(t)^{-1}(x(t)-y(t)) \in A (y(t))$, we have that, for $t$ large enough, say $t\geq t_0$
  \begin{align}
 d(y(t),S)=   \norm{y(t)-y^*(t)}\leq \kappa\norm{v(t)}.
  \end{align}
 Hence
  \begin{align*}
    \norm{x(t)- x^*(t)}
    \leq \norm{x(t)-y^*(t)}
    &\leq\norm{x(t)-y(t)}+\norm{y(t)-y^*(t)} \\
    & \leq \norm{x(t)-y(t)}+\kappa\norm{v(t)}\\
    &=\norm{x(t)-y(t)}\left(1+\dfrac{\kappa}{\lambda (t)}\right) .
  \end{align*}
 Take $\alpha\in(0,1)$. Since $\lambda(t)\to\infty$ as $t\to\infty$,
   for $t$ large enough
   \begin{align}\label{eq:errorbound1}
  \norm{x(t) - x^*(t)}\leq \alpha^{-1}\norm{x(t) - y(t)} .
   \end{align}
  Define
  \begin{align*}
    g(t):=\frac{1}{2} d^2(x(t),S)   =   \frac{1}{2}\norm{x(t)-x^*(t)}^2.
  \end{align*}
Using successively the classical derivation chain  rule, and 
\eqref{eq:edomm}, we obtain
  \begin{align*}
g'(t)&= \left\langle  x(t) - x^*(t), \dot{x} (t)\right\rangle\\
 &= -\left\langle  x(t) - x^*(t), x(t) -y(t)\right\rangle \\  
  &= -\norm{x(t)-y(t)}^2   - \left\langle  y(t) - x^*(t), x(t) -y(t)\right\rangle . 
   \end{align*}  
By the monotonicity of $A$, and    $\lambda(t)^{-1}(x(t)-y(t)) \in A (y(t))$, $0 \in A (x^*(t))$, we have
\begin{align*}
    \left\langle  y(t) - x^*(t), x(t) -y(t)\right\rangle \geq 0.
  \end{align*}   
Combining the two above inequalities, we obtain   
\begin{align}\label{eq:errorbound2}
g'(t)\leq -\norm{x(t)-y(t)}^2   . 
   \end{align}    
From \eqref{eq:errorbound1}, \eqref{eq:errorbound2}, and the definition of $g$, we infer
  \begin{align*}
    g'(t)\leq -2{\alpha}^2 g(t),
  \end{align*}
  and it follows from Gronwall's lemma that $g(t)\leq ce^{-2{\alpha}^2 t}$,
  which proves the first inequality.\\
   To prove the second inequality, we use the inequality 
  \begin{align}\label{eq:errorbound3}
  \norm{x(t)-y(t)} \leq d(x(t),S)
   \end{align} 
   which is a direct consequence of the $\frac{1}{\lambda}$-Lipschitz continuity of $A_{\lambda}$. For $z\in S$, since $A_{\lambda (t)}z=0$
 $$
 \norm{A_{\lambda (t)} x(t)} = \norm{A_{\lambda (t)} x(t)- A_{\lambda (t)}z }    \leq \frac{1}{\lambda (t)}\norm{x(t)-z} .
 $$ 
Equivalently,   $\norm{x(t)-y(t)} \leq \norm{x(t)-z}$ for all $z\in S$, which gives \eqref{eq:errorbound3}.  
Then  use the first inequality,  and the
  equality $\lambda(t)\norm{x(t)-y(t)}=\theta$, and so obtain the second inequality.\\
   The third inequality
  follows from the second one, and the equality
  $\lambda(t)^2\norm{v(t)}=\theta$.\\
   To prove the last inequality,
  observe that for $t_1<t_2$,
  \begin{align*}
    \norm{x(t_2)-x(t_1)}\leq\int_{t_1}^{t_2}\norm{\dot x (t)}dt =\int_{t_1}^{t_2}\norm{x (t)- y (t)}dt
    \leq \int_{t_1}^{t_2} d(x(t),S)dt
  \end{align*}
where the last inequality comes from \eqref{eq:errorbound3},  and the strong convergence of $x(t)$, as well as the last inequality follows.
\end{proof}

\begin{remark} \rm{
In the Appendix, in the case of an isotropic linear monotone operator, we can perform an explicit computaion of $x$, $\lambda$, and observe that their rate of convergence  are in accordance with the conclusions of Theorem \ref{thm-error-bound}.}
\end{remark}

\subsection{Weak versus strong convergence}
A famous counterexample due to Baillon \cite{Ba} shows that the trajectories of the steepest descent dynamical system associated to a convex potential can converge weakly but not strongly. The existence of such a counterexample for  \eqref{eq:edomm} is an interesting open question, whose study goes beyond this work. In the following theorem, we provide some practically important situations where the strong convergence holds for  system \eqref{eq:edomm}.

\begin{theorem} \label{thm-strg-conv}
 Assuming $S=A^{-1}(0)$ is non-empty, then $x(t)$ converges strongly to some  $x^*\in A^{-1}(0)$, in the following situations:

 i)  $A$ is strongly monotone;
 
 ii) $A = \partial f$, where $f: \mathcal H \to \mathbb R \cup + \left\{\infty\right\}$ is a proper closed convex function, which is boundedly inf-compact;
 
 iii) $S = A^{-1} (0)$ has a nonempty interior. 

\end{theorem}
\begin{proof}
$i)$ If $A^{-1}$ is Lipschitz continuous at $0$, then  assumption \textbf{A0)} holds, and, by Theorem \ref{thm-error-bound}, each trajectory $x(t)$ of \eqref{eq:edomm} converges strongly to some  $x^*\in A^{-1}(0)$. In particular, if $A$ is strongly monotone, i.e.,
there exists a positive constant $\alpha$ such that for any $y_i \in Ax_i$, $i=1,2$
$$
 \left\langle y_2 - y_1,  x_2 -x_1 \right\rangle \geq \alpha   \|  x_2 - x_1 \|^2,
$$
then $A^{-1}$ is Lipschitz continuous.
In that case, $A^{-1} (0)$ is reduced to a single element $z$, and each trajectory $x(t)$ of \eqref{eq:edomm}  converges strongly to $z$, with the rate of convergence given by Theorem \ref{thm-error-bound}.

$ii)$ $A = \partial f$, where $f: \mathcal H \to \mathbb R \cup + \left\{\infty\right\}$ is a proper closed convex function, which is supposed to be boundedly inf-compact,
 i.e., for any $R>0$ and $l\in \mathbb R$,
  $$\left\{ x\in \mathcal H:  \ f(x)\leq l,  \mbox{ and }  \|x \| \leq R \right\}  \  \mbox{ is relatively compact in} \  \mathcal H.$$
By Corollary \ref{cr:dfx}, \ $ t \mapsto f(x(t))$ is non-increasing, and $x(\cdot)$ is contained in a sublevel set of $f$. Thus, the orbit  $x(\cdot)$ is relatively compact, and converges weakly.
Hence, it converges strongly.

$iii)$ Suppose now that $S = A^{-1} (0)$ has a nonempty interior. Then there $r >0$ and $p \in  A^{-1} (0)$ such that the ball $B (p,r)$ of radius $r$ centered at $p$ is contained in $S$. For any given  $\lambda >0$, we have
$
A^{-1} (0) = A_{\lambda}^{-1} (0).
$
Hence, for any  $\lambda >0$, we have $B(p,r) \subset A_{\lambda}^{-1} (0)$.
By the monotonicity property of $A_{\lambda}$, for any $\xi \in \mathcal H$, $\lambda >0$, and $h \in \mathcal H$ with $\|h\| \leq 1$,
$$
\left\langle  A_{\lambda}(\xi ), \xi -( p + rh)\right\rangle \geq 0 .
$$
Hence
 \begin{equation}\label{eq:interior1}
r \| A_{\lambda}(\xi ) \| = r \sup_{\| h \| \leq 1} 
\left\langle  A_{\lambda}(\xi ),  h)\right\rangle
\leq \left\langle  A_{\lambda}(\xi ), \xi - p )\right\rangle .
 \end{equation}
The edo \eqref{eq:edomm} can be written as
$ \dot x(t) +  \lambda(t) A_{\lambda (t)}x(t) =0  $.
Taking $\lambda = \lambda (t)$, and $\xi = x(t)$ in \eqref{eq:interior1}, we obtain 
$$
\|\dot x(t)\| = \lambda(t) \| A_{\lambda (t)}(x(t) ) \| \leq \frac{\lambda(t)}{r}\left\langle  A_{\lambda (t)}(x(t) ), x(t) - p )\right\rangle 
 $$
Using again \eqref{eq:edomm} we obtain
 \begin{equation}\label{eq:interior2}
\|\dot x(t)\|
\leq - \frac{1}{r}\left\langle  \dot x(t), x(t) - p )\right\rangle .
 \end{equation}
The end of the proof follows standard arguments, see for example \cite[Proposition 60]{PS}. Inequality \eqref{eq:interior2} implies, for any $0 \leq s \leq t$
\begin{align*}
\| x(t) - x(s)\| &\leq \int_s^t  \|\dot x(\tau)\| d\tau \\
&\leq - \frac{1}{r}\int_s^t  \left\langle  \dot x(\tau), x(\tau) - p \right\rangle d\tau\\
&\leq \frac{1}{2r} ( \| x(s)-p\|^2 - \| x(t)-p\|^2 ).
\end{align*}
 By Theorem \ref{th:main.2} $iii)$,  $ \| x(t) -p \|$ is convergent. As a consequence, the trajectory $x(\cdot)$ has the Cauchy property in the Hilbert space $\mathcal H$, and hence converges strongly.
\end{proof}

\section{A link with the regularized Newton system}
\label{sec:lk}


In this section, we  show how the dynamical system~\eqref{eq:edomm} is
linked with the regularized Newton system proposed and analyzed in
 \cite{AAS}, \cite{ARS}, \cite{AS}. Given $x_0\notin A^{-1} (0)$, let
 $(x,\lambda):[0,+ \infty[\to {\mathcal H}\times \R_{++}$  be the unique global solution   of
  the Cauchy problem \eqref{eq:edomm-cauchy}. For any $t\geq 0$ \ define
\begin{align}
  \label{eq:lk0}
  y(t)=(I + \lambda(t) {A})^{-1}x(t),\qquad v(t)=\frac{1}{\lambda (t)}(x(t)-y(t)).
\end{align}
We are going to show that $y(\cdot)$ is solution of a regularized Newton system. For proving this result, we first establish some further properties satisfied by $y(\cdot)$.
\begin{proposition}
  \label{pr:oode}
  For $y(\cdot)$ and $v(\cdot)$ as defined in \eqref{eq:lk0} it holds that
  
  \smallskip
  
 i)   $v(t)\in Ay(t)$, $\lambda (t)  v(t)+y(t)-x(t)=0$, and $\dot x (t)=y(t)-x(t)$ \ for all $t\geq 0$;
 
   \smallskip

ii) $v(\cdot)$ and $y(\cdot)$ are locally Lipschitz continuous;

  \smallskip

iii)    $\dot y (t) +\lambda (t) \dot v (t)+(\lambda (t)+\dot\lambda (t))v(t)=0$  \ for almost all $t\geq 0$;

  \smallskip

iv) $\inner{\dot y (t)}{\dot v (t) }\geq 0$ and $\inner{\dot y (t)}{v (t)}\leq 0$ \ for almost all $t\geq 0$;

  \smallskip

 v)   $\norm{v (\cdot)}$ is non-increasing.

\end{proposition}

\begin{proof}
  Item $i)$ follows trivially from \eqref{eq:lk0} and \eqref{eq:edomm}.
  Item $ii)$ follows from the local Lipschitz continuity of $\lambda$, and
  the properties of the resolvent, see Proposition \ref{resolv. cont}.
  Hence $x,y, \lambda,v$ are differentiable almost everywhere. By differentiating
  $\lambda v+y-x=0$, and using	$\dot x=y-x$, we obtain item $iii)$.
  To prove item $iv)$, assume that $y$
  and $v$ are differentiable at $t \geq 0$.  It follows from the
  monotonicity of $A$ and the first relation in item $iv)$ that
  if $t' \neq t$ and $t' \geq 0$
  \begin{align*}
    \dfrac{\inner{y(t')-y(t)}{v(t')-v(t)}}{(t'-t)^2}\geq 0.
  \end{align*}
  Passing to the limit as $t'\to t$ in the above
  inequality, we conclude that the first inequality in item $iv)$ holds.  To prove the last inequality, assume that $\lambda (\cdot)$ is also
  differentiable at $t$. Using item $iii)$, after scalar multiplication by $\dot y (t)$, we obtain
  \[
  \norm{\dot y (t)}^2+\lambda (t)\inner{\dot y (t)}{\dot
    v (t)}+(\lambda (t)+\dot\lambda (t))\inner{\dot y (t)}{v(t)}=0.
  \]
  To end the proof of item $iv)$, note that  $\dot \lambda(t)\geq 0$ (by Theorem \ref{th:main.1}, $ii)$, $\lambda (\cdot)$
  is non-decreasing), and use the first
  inequality of item $iv)$.
  In view of \eqref{eq:lk0} and \eqref{eq:edomm},
  $\lambda^2 (t)\norm{v (t)}=\theta$ for all $t \geq0$. This result, together with
  Lemma~\ref{lm:lnond} proves item $v)$.
\end{proof}
Hence (almost everywhere) $y(\cdot)$ and $v (\cdot)$, as defined in \eqref{eq:lk0}, satisfy
the differential inclusion
\begin{equation}\label{{eq:old}}
\left\{
\begin{array}{l}
v (t)\in {A}y(t);    \\
\rule{0pt}{20pt}
\dot y (t)+\lambda (t)\dot v (t)+(\lambda (t)+\dot\lambda (t))v(t)=0.
\end{array}\right.
\end{equation}
Recall that $\lambda (\cdot)$ is locally absolutely continuous, and satisfies almost everywhere
\[
0\leq\dot \lambda (t) \leq \lambda (t).
\]
Let us consider the time rescaling defined by
\begin{align}
  \label{eq:tau}
  \tau(t)=\int_0^t\dfrac{\lambda(u)+\dot\lambda(u)}{\lambda(u)}du=
 t+\ln (\lambda(t)/\lambda(0)).
\end{align}
Since $ 1 \leq \dfrac{\lambda(u)+\dot\lambda(u)}{\lambda(u)} \leq  2 $,       we have $t \leq \tau(t) \leq 2t$.
Hence $t\mapsto  \tau(t)$ is a monotone  function which increases from $0$ to $+\infty$ as $t$ grows from  $0$ to $+\infty$. The link with  the regularized Newton system  is made precise in the following statement.
\begin{theorem} \label{thm-reg-Newton}  For $y(\cdot)$ and $v(\cdot)$ as defined in \eqref{eq:lk0},
let us set \ $y(t) = \tilde{y}(\tau (t)), \quad  v(t) = \tilde{v}(\tau (t))$,
where the time rescaling is given by  $\tau(t)= \int_0^t\dfrac{\lambda(u)+\dot\lambda(u)}{\lambda(u)}du$.
Then, $ (\tilde{y}, \tilde{v})$ is solution of the regularized Newton system
\begin{equation}\label{eq:lk2}
\left\{
\begin{array}{l}
\tilde{v}\in {A}\tilde{y};    \\
\rule{0pt}{20pt}
\frac{1}{\lambda \circ {\tau}^{-1}} \dfrac{d}{d\tau}\tilde{y} + \dfrac{d}{d\tau}\tilde{v} +\tilde{v} =0.
\end{array}\right.
\end{equation}
\end{theorem}
That's the regularized Newton system which has been studied in \cite{AS}. The (Levenberg-Marquardt) regularization parameter is equal to $\frac{1}{\lambda \circ {\tau}^{-1}}$. Since $\lambda (t)$ tends to infinity,
the regularization parameter converges to zero as $\tau$ tends to infinity. This makes our system asymptotically close to the Newton method. We may expect fast convergence properties.
That's precisely the subject of the next section. Let us complete this section with the following relation allowing to recover $x$ from $y$.

\begin{lemma}
  \label{lm:cc}
  For any $t_2>t_1\geq 0$,
  \begin{align*}
 x(t_2)& =\int_0^{\Delta t} \left[ (1-e^{-\Delta t})y(t_1+u) +
    e^{-\Delta t}x(t_1)\right]\dfrac{e^u}{e^{\Delta t}-1}\;du.
  \end{align*}
  where $\Delta t=t_2-t_1$.
\end{lemma}

\begin{proof}
  It suffices to prove the equality for $t_1=0$ and $t_2=t=\Delta t$.
  Since $\dot x=y-x$, trivially $\dot x+x=y$. So
  \[
  e^tx(t)-x_0=\int_0^te^uy(u)\,du.
  \]
  Whence
  \begin{align*}
    x(t)&=e^{-t}x_0+e^{-t}\int_0^te^uy(u)\;du\\
    &=e^{-t}\int_0^te^u\left[(e^t-1)y(u)+x_0\right]\dfrac{1}{e^t-1}\;du
  = \int_0^t\left[(1-e^{-t})y(u)+e^{-t}x_0\right]\dfrac{e^u}{e^t-1}\;du
  \end{align*}
  which is the desired equality.
\end{proof}

 \section{The subdifferential case}
 \label{sec:sbd}
From now on, in this section, we assume that $A = \partial f$, where $f: \mathcal H \to \mathbb R \cup + \left\{\infty\right\}$ is a proper closed convex function.
 Let us recall the generalized derivation chain rule from Br\'ezis~\cite{Br} that will be useful:
\begin{lemma}{\rm {\cite[Lemme 4, p.73]{Br}}} \label{lemma_Brezis} Let	$\Phi: \mathcal H \rightarrow \R \cup \{+\infty\}$ be a closed convex proper function. Let
  $u\in L^2(0, T; \mathcal H)$ be such that $\dot{u}\in L^2 (0, T; \mathcal H)$, and
  $u(t)\in {\rm \mbox{dom}}(\partial \Phi)\ \mbox{for a.e.} \ t$. Assume that there exists
  $\xi\in L^2(0, T; \mathcal H)$ such that $\xi(t)\in \partial \Phi(u(t))$ for
  a.e. $t$. Then the function $t\mapsto \Phi(u(t))$ is absolutely
  continuous, and for every $t$ such that $u$ and  $\Phi(u)$ are differentiable at $t$, and $u(t)\in {\rm \mbox{dom}}(\partial \Phi)$, we have
$$
\forall h\in  \partial \Phi(u(t)),\qquad \frac{d}{dt}  \Phi(u(t))=\langle \dot{u}(t), h\rangle.
$$
\end{lemma}

\subsection{Minimizing property}
Since $v(t)\in\partial f(y(t))$, $\lambda (t) v(t)=x(t)-y(t)$, and
$\lambda(t)^2\norm{v(t)}=\theta$, by the convex subdifferential inequality
\begin{align}
  \nonumber
  f(x(t))\geq f(y(t))+\inner{x(t)-y(t)}{v(t)}
  & \geq f(y(t))+\lambda(t)\norm{v(t)}^2 \\
  \label{eq:x}
  & =	 f(y(t))+\sqrt{\theta}\norm{v(t)}^{3/2}.
\end{align}

\begin{lemma}
  \label{lm:dfy}
  The function $t\mapsto f(y(t))$ is locally Lipschitz continuous,
  non-increasing and for any $t_2 > t_1 \geq 0$,
  \begin{align}\label{eq: basic1}
  f(x(t_2)) & \leq \int_0^{\Delta t} \left[ (1-e^{-\Delta t})f(y(t_1+u)) +
    e^{-\Delta t}f(x(t_1))\right]\dfrac{e^u}{e^{\Delta t}-1}\;du\\
  & \leq  (1-e^{-\Delta t})f(y(t_1)) +
    e^{-\Delta t}f(x(t_1))
  \end{align}
  where $\Delta t=t_2-t_1$.
\end{lemma}

\begin{proof}
  Suppose that $t_2,t_1\geq 0$, $t_1\neq t_2$ and let
  \[
  y_1=y(t_1),\;  v_1=v(t_1),\;y_2=y(t_2),\; v_2=v(t_2).
  \]
  Since $v_i\in\partial f(y_i)$ for $i=1,2$
  \begin{align*}
    f(y_2)\geq f(y_1)+\inner{y_2-y_1}{v_1},\quad
    f(y_1)\geq f(y_2)+\inner{y_1-y_2}{v_2}.
  \end{align*}
  Therefore
  \begin{align*}
   \inner{y_2-y_1}{v_1} \leq f(y_2)-f(y_1) \leq \inner{y_2-y_1}{v_2}
  \end{align*}
  and
  \[
  \abs{f(y_1)-f(y_2)}\leq\norm{y_1-y_2}\max\{\norm{v_1},\norm{v_2}\}\leq
  \norm{y_1-y_2}\norm{v(0)}
  \]
  where in the last inequality, we use that
    $\norm{v (\cdot)}$ is non-increasing, (see Proposition \ref{pr:oode}, item $v)$).
  Since $t\mapsto y(t)$ is locally Lipschitz continuous, $t\mapsto f(y(t))$ is also
  locally Lipschitz continuous on $[0,\infty[$. Moreover, $t\mapsto f(y(t))$ is
  differentiable almost everywhere. Since $y$ is locally Lipschitz continuous, and $v(\cdot)$ is bounded,
   by Lemma \ref{lemma_Brezis},  the derivation chain rule holds true (indeed, it provides another proof of the absolute continuity of $t\mapsto f(y(t))$). Hence
  \begin{align*}
    \dfrac{d}{dt}f(y(t))=\inner{\dot y(t)}{v(t)} \leq 0,
  \end{align*}
  where in the last inequality, we use Proposition \ref{pr:oode}, item $iv)$.  Hence $t\mapsto f(y(t))$ is locally Lipschitz continuous, and
  non-increasing.
  Let us now prove  inequality \eqref{eq: basic1}. Without any restriction we can take	$t_1=0$ and $t_2=t=\Delta t$.
  By Lemma \ref{lm:cc}
  \begin{equation}
    \label{eq:cc2}
    x(t)= \int_0^t\left[(1-e^{-t})y(u)+e^{-t}x_0\right]\dfrac{e^u}{e^t-1}\;du.
  \end{equation}
 The conclusion follows from the convexity of $f$, Jensen's inequality,
  and $t\mapsto f(y(t))$
  non-increasing.
\end{proof}
\begin{corollary}
  \label{cr:dfx}
  If $f(x(0))<+\infty$, then for any $t\geq 0$, we have
  \begin{align}
    &i) \ f(x(t))< + \infty,\\
    &ii) \ t \mapsto f(x(t)) \ \mbox{ is non-increasing},\\
    &iii)\  \limsup_{h\to 0^+} \dfrac{f(x(t+h))-f(x(t))}{h}\leq f(y(t))-f(x(t))
    \leq -\sqrt{\theta}\norm{v(t)}^{3/2}.
  \end{align}
\end{corollary}
\begin{proof}
  Take $t \geq 0$ and $h > 0$. Direct use of Lemma~\ref{lm:dfy} with
  $t_1 = t$ and $t_2 = t+h$ yields
  \begin{align*}
    \dfrac{f(x(t+h))-f(x(t))}{h} \leq \dfrac{1-e^{-h}}{h}(f(y(t))-f(x(t))),
  \end{align*}
  and the conclusion follows by taking the $\limsup$ as $h \to 0^+$ on both
  sides of this inequality, and by using \eqref{eq:x}.
\end{proof}

\subsection{Rate of convergence}

In this subsection, we assume that $f$ has  minimizers. Let
\[
\bar z\in\arg\min f,\qquad d_0=\inf
\{ \norm{x_0 -z} :\;z\text{ minimizes } f\}= \norm{x_0- \bar z}.
\]
Since $v(t)\in\partial f(y(t))$, for any $t \geq 0$
\begin{align*}
  f(y(t))-f(\bar z) \leq \inner{y(t)-\bar z}{v(t)}
  &\leq \norm{y(t)-\bar z}\norm{v(t)}\\
   & \leq \norm{x(t)-\bar z}\,\norm{v (t)} \leq d_0\norm{v (t)}
\end{align*}
where we have used  $y (t)= J_{\lambda (t)}^{A}(x(t))$, $\bar z= J_{\lambda (t)}^{A}(\bar z)$,	$J_{\lambda (t)}^{A}$ nonexpansive,
 and $t\mapsto \norm{x(t)-\bar z}$ non-increasing (see (\ref{eq:hd})).
Combining the above inequality with \eqref{eq:x}, we conclude that for
any $t\geq0$
\begin{align}
  \label{eq:c0}
  f(x(t))\geq f(y(t))+(f(y(t)-f(\bar z))^{3/2}\sqrt{\theta/d_0^3}.
\end{align}
Now we will use the following auxiliary result, a direct consequence of
 the convexity property of $r \mapsto r^{3/2}$.
\begin{lemma}
  \label{lm:aux1}
  If $a,b,c\geq 0$ and $a\geq b+c b^{3/2}$ then
  \begin{align*}
    b\leq a-\dfrac{c a^{3/2}}{1+(3c/2)
   a^{1/2}}
  \end{align*}
\end{lemma}

 \begin{proof}
   The non-trivial case is $a,c>0$, which will be analyzed. Define
   \[
   \varphi:[0,\infty)\to \R,\qquad  \varphi(t)=t+ct^{3/2}.
   \]
   Observe that $\varphi$ is convex, and $a \geq \varphi (b)$. Let us write the convex differential inequality at $a$
   \begin{align*}
    a \geq \varphi (b) \geq \varphi (a) + \varphi'(a) (b-a).
  \end{align*}
   After simplification, we obtain the desired result.
 \end{proof}

\begin{proposition}
  \label{pr:c1}
  For any $t\geq0$,
  \begin{align}
    f(y) \leq f(x)-
    \dfrac{\kappa(f(x)-f(\bar z))^{3/2}}{1+(3\kappa/2)(f(x)-f(\bar z))^{1/2}},
  \end{align}
  where $x=x(t)$, $y=y(t)$ and $\kappa=\sqrt{\theta/d_0^3}$.
\end{proposition}

\begin{proof}
  Subtracting $f(\bar z)$ on both sides of \eqref{eq:c0} we conclude that
  \begin{align*}
    f(x(t))-f(\bar z)\geq f(y(t))-f(\bar z)+(f(y(t)-f(\bar
    z))^{3/2}\sqrt{\theta/d_0^3}.
  \end{align*}
  To end the proof, use Lemma~\ref{lm:aux1} with $a=f(x(t))-f(\bar z)$,
  $b=f(y(t))-f(\bar z)$ and $c=\sqrt{\theta/d_0^3}$.
\end{proof}

\begin{theorem}
  \label{th:complexity}
  Let us assume that $f(x(0)) < + \infty$. Set $\kappa=\sqrt{\theta/d_0^3}$. Then,
  for any $t \geq 0$
  \begin{align*}
    f(x(t))-f(\bar z)\leq \dfrac{f(x_0)-f(\bar z)}{
      \left[
	1+\dfrac{t\kappa\sqrt{f(x_0)-f(\bar z)}}{2+3\kappa\sqrt{f(x_0)-f(\bar z)}}
      \right]^2} 
  \end{align*}
\end{theorem}

\begin{proof}
 Set $\beta(t):=f(x(t))-f(\bar z)$. Consider first the case where $\beta (\cdot)$ is locally Lipschitz
continuous.
  Combining Proposition~\ref{pr:c1}
  with Corollary~\ref{cr:dfx}, and taking into account that $f(x(\cdot))$ is
  non-increasing,  we conclude that, almost everywhere
  \begin{align*}
   \dfrac{d}{dt}\beta \leq -\dfrac{\kappa \beta^{3/2}}{1+(3\kappa/2)\beta^{1/2}}
    \leq -\dfrac{\kappa \beta^{3/2}}{1+(3\kappa/2)\beta_0^{1/2}}
  \end{align*}
  where $\beta_0=\beta(0)=f(x_0)-f(\bar z)$. Defining
  \begin{align*}
    u=1/\sqrt{\beta},\qquad \tilde\kappa=
    \dfrac{\kappa}{1+(3\kappa/2)\beta_0^{1/2}}
  \end{align*}
  and substituting $\beta=1/u^2$ in the above inequality, we conclude that
  \begin{align*}
    -2u^{-3}\dfrac{d}{dt}u\leq - \dfrac{\kappa u^{-3}}{1+(3\kappa/2)\beta_0^{1/2}}.
  \end{align*}
  Therefore, for any $t\geq 0$,
  \begin{align*}
    u(t) \geq  \dfrac{t\kappa}{2+3\kappa\beta_0^{1/2}}+1/\beta_0^{1/2}.
  \end{align*}
  To end the proof, substitute $u=1/\sqrt{\beta}$ in the above inequality.
  In the general case, without assuming $\beta$ locally Lipschitz, we can write the differential equation in terms of differential measures
  ($\beta$ is non-increasing, hence it has a bounded variation, and its distributional derivative is a Radon measure):
  \begin{align*}
   d\beta  + \dfrac{\kappa }{1+(3\kappa/2)\beta_0^{1/2}} \beta^{3/2} \leq 0.
  \end{align*}
Let us regularize this equation  by convolution, with the help of a smooth kernel $\rho_{\epsilon}$ (note that we use convolution in $\mathbb R$, whatever the dimension of $\mathcal H$, possibly infinite). By convexity of $r \mapsto r^{3/2}$, and Jensen inequality, we obtain that $\beta \ast \rho_{\epsilon}$ is a smooth function that still satisfies the differential inequality. Thus we are reduced to the preceding situation, with bounds which are independent of $\epsilon$, whence the result by passing to the limit
  as $\epsilon \to 0.$
\end{proof}
Let us complete the convergence analysis by the following  integral estimate.

\begin{proposition}\label{rate-conv-y}
Suppose $S= \arg\min f \neq \emptyset$. 
Then
$$
\int_0^{+\infty} \lambda (t) (f(y(t)) - \inf f) dt \leq \frac{1}{2} \mbox{dist}^2 (x_0, S) .
$$
\end{proposition}
\begin{proof}
Let us return to the proof of Theorem \ref{th:main.2}, with  $A = \partial f$. Setting $ h_z (t) :=	\frac{1}{2} \| x(t) - z\|^2$, with $z\in \arg\min f$, by (\ref{basic-rate}) we have 
  \begin{align}\label{basic-rate-2}
    \dot{h}_z (t) +\inner{y(t) -z}{\lambda (t)v(t)} \leq 0.
  \end{align}
By the convex subdifferential inequality, and $v(t) \in \partial f (y(t))$, we have  
 $$
 f(z) \geq f(y(t)) + \inner{z- y(t)}{v(t)}.
 $$ 
Combining the two above inequalities, we obtain  
 \begin{align}\label{basic-rate-3}
    \dot{h}_z (t) + \lambda (t) (f(y(t)) - \inf f) \leq 0.
  \end{align} 
By integrating this inequality, we obtain the announced result. 
\end{proof}


\section{A large-step proximal point method for convex optimization with
  relative error tolerance}
\label{sec:ur.pp}
In this section, we study the iteration complexity of a variant of the proximal point (PP) method for convex optimization (CO).
It can be viewed as a discrete version of the  continuous dynamical system studied in the previous sections.
The main distinctive
features of this variant are: a relative error tolerance for the solution
of the proximal subproblems similar to the ones proposed in~\cite{So-Sv1,So-Sv2};
a large-step condition, as proposed in~\cite{MS1,MS2}.

The PP method~\cite{MR0298899,Rock,MR0418919}
is a classical method for finding zeroes of
maximal monotone operators and, in particular,
for solving CO problems. It has been used as a framework for the analysis and design of
many practical algorithms (e.g., the augmented Lagrangian, the  proximal-gradient, or the alternating proximal minimization algorithms).
The fact that its classical convergence analysis~\cite{Rock} requires
the errors to be summable, motivates
the introduction in~\cite{So-Sv1,So-Sv2} of the Hybrid Proximal Extragradient (HPE) method,
an inexact PP type method which allows relative error tolerance in the solution of the proximal subproblems.
The relative error tolerance of the HPE was also used for 
minimization of semi-algebraic, or tame functions in~\cite{ABS}.

Consider the convex optimization problem:
\begin{align}
 \label{eq:p2}
	  \mbox{minimize}\,f(x)\quad \mbox{s.t.}\quad x\in \HH,
\end{align}
where $f:\HH\to \R\cup\{+\infty\}$ is a (convex) proper and closed function.
An exact proximal point iteration at $x\in \HH$
with stepsize $\lambda>0$
consists in computing
\begin{align*}
  y=(I+\lambda\partial f)^{-1}(x).
\end{align*}
Equivalently, for a given pair $(\lambda, x)\in \R_{++}\times \HH$, we
have to compute $y\in \HH$ such that
\begin{align*}
  0\in \lambda \partial f(y)+y-x.
\end{align*}
Decoupling the latter inclusion,
we are led to the following
\emph{proximal inclusion-equation system}:
\begin{align}
  \label{eq:ps}
  \begin{aligned}
   v\in\partial f(y),\quad \lambda v+y-x=0.
  \end{aligned}
\end{align}
We next show how errors in both the inclusion and
the equation in~\eqref{eq:ps}
can be handled with an appropriate error
criterion ($\partial_\varepsilon f$ stands for the classical notion of Legendre-Fenchel $\epsilon$-subdifferential).

\begin{proposition}
  \label{pr:ie}
  Let $x\in \HH$, $\lambda>0$ and $\sigma\in[0,1[$.
  If $y,v\in\HH$ and $\varepsilon \geq 0$ satisfy the conditions
  \begin{align}
    \label{eq:is}
    v\in\partial_\varepsilon f(y),\quad\norm{\lambda v+y-x}^2+2\lambda
    \varepsilon \leq \sigma^2\norm{y-x}^2,
  \end{align}
  then, the following statements hold:
  \begin{enumerate}
  \item[(a)]\label{it:b1}
    $ f(x')\geq f(y)+\inner{v}{x'-y}
    -\varepsilon \qquad \forall x'\in \HH$;
  \item[(b)] \label{it:b2}
    $ \displaystyle f(x)\geq f(y)+
    \frac{\lambda}{2} \norm{v}^2+\frac{1-\sigma^2}{2\lambda}
    \norm{y-x}^2\geq f(y)$;
  \item[(c)]\label{it:b3}
    $ (1+\sigma)\norm{y-x} \geq \norm{\lambda v}
    \geq (1-\sigma)\norm{y-x}$;
  \item[(d)]\label{it:b4} $\displaystyle
    \varepsilon\leq \frac{\sigma^2}{2(1-\sigma)}
    \norm{v}\,\norm{y-x}$;
  \end{enumerate}
  and
  \begin{align}
	\label{eq:pr:ie}
     \frac{\lambda}{2} \norm{v}^2+\frac{1-\sigma^2}{2\lambda}
    \norm{y-x}^2 \geq
 \max \left\{
      \norm{v}^{3/2}\sqrt{\lambda\norm{y-x}(1-\sigma)},
      \frac{1-\sigma}{\lambda} \norm{y-x}^2 \right\}  .
  \end{align}
\end{proposition}
\begin{proof}
\textit{(a)}
This statement follows trivially from the inclusion in~\eqref{eq:is},
and the definition of $\varepsilon$-subdifferentials.

\textit{(b)} First note that
  the inequality in~\eqref{eq:is}
  is equivalent to
  \[
  \norm{\lambda v}^2+\norm{y-x}^2
  -2\lambda\left[\inner{v}{x-y}-\varepsilon\right]\leq \sigma^2
  \norm{y-x}^2.
  \]
  Dividing both sides of the latter inequality by $2\lambda$,
	and using some trivial algebraic manipulations, we obtain
	\begin{align*}
    \inner{v}{x-y}-\varepsilon\geq&
   \frac{\lambda}{2}
   \norm{v}^2+
   \frac{1-\sigma^2}{2\lambda}
   \norm{y-x}^2,
  \end{align*}
	which, in turn, combined with \textit{(a)} evaluated at $x'=x$, yields
  the first inequality in \textit{(b)}. To complete the proof of \textit{(b)}, note that the second inequality
  follows trivially from the assumptions that $\lambda>0$ and
  $0\leq\sigma<1$.

 \textit{(c)} Direct use of the triangle
  inequality yields
  \begin{align*}
   \norm{y-x}+\norm{\lambda v+y-x} \geq \norm{\lambda v} \geq
   \norm{y-x}-\norm{\lambda v +y-x}.
  \end{align*}
	Since $\sigma\geq 0$, $\lambda>0$, and $\varepsilon\geq 0$, it follows
  from~\eqref{eq:is} that $ \norm{\lambda v+y-x}\leq\sigma\norm{y-x}$,
	which in turn combined with the latter displayed equation
	proves \textit{(c)}.

\textit{(d)} In view of the inequality in~\eqref{eq:is}, the second inequality in
  \textit{(c)}, and the assumption that $\sigma<1$, we have
	\begin{align*}
    2\lambda\varepsilon \leq \sigma^2\norm{y-x}^2
    & \leq \dfrac{\sigma^2}{1-\sigma}\norm{\lambda v}\,\norm{y-x},
  \end{align*}
	which trivially gives the statement in \textit{(d)}.

	To complete the proof of the proposition, it remains to prove~\eqref{eq:pr:ie}.
  To this end,  first note  that, due to \textit{(c)}, we have $y-x=0$
	if and only if $v=0$, in which case~\eqref{eq:pr:ie} holds trivially.
	Assume now that $y-x$ and $v$ are nonzero vectors. Defining
	the positive scalars $\theta=\lambda\norm{y-x}$, $\mu=\lambda\norm{v}/\norm{y-x}$
	and using \textit{(c)} we conclude that
	\begin{equation}
    1-\sigma\leq \mu\leq 1+\sigma.
    \label{eq:wk}
  \end{equation}
	Moreover, it follows directly from the definitions of $\theta$ and $\mu$ that
	\begin{align*}
	\frac{\lambda}{2} \norm{v}^2+\frac{1-\sigma^2}{2\lambda}
	\norm{y-x}^2
	& =  \frac{\lambda}{2} \norm{ v }^2
	\left(1+\frac{1-\sigma^2}{\mu^2}\right)
	=
    \norm{ v}^{3/2}\frac{\sqrt{\theta \mu}}{2}
    \left(1+\frac{1-\sigma^2}{\mu^2}\right).
  \end{align*}
  Since $t+1/t\geq 2$ for every $t>0$, it follows that
  \begin{align*}
    \sqrt{\mu}\left(1+\frac{1-\sigma^2}{\mu^2}\right) =
    \sqrt{\dfrac{1-\sigma^2}{\mu}}\left(
      \dfrac{\mu}{\sqrt{1-\sigma^2}}+
      \dfrac{\sqrt{1-\sigma^2}}{\mu}\right)
    \geq 2 \sqrt{\dfrac{1-\sigma^2}{\mu}}\geq 2\sqrt{1-\sigma},
  \end{align*}
  where the second inequality follows from the upper bound for $\mu$
  in~\eqref{eq:wk}.
	Combining the last two displayed equations, and using the definition of $\theta$, we obtain
  \begin{align*}
	\frac{\lambda}{2} \norm{v}^2+\frac{1-\sigma^2}{2\lambda}
	\norm{y-x}^2
	& \geq \norm{v}^{3/2}\sqrt{\lambda\norm{y-x}(1-\sigma)}.
  \end{align*}
	Likewise, using the second inequality \textit{(c)} we obtain
		\[
	\frac{\lambda}{2} \norm{v}^2+\frac{1-\sigma^2}{2\lambda}
	\norm{y-x}^2
	\geq
    \frac{(1-\sigma)^2}{2\lambda} \norm{ y - x }^2
    +
    \frac{1-\sigma^2}{2\lambda} \norm{ y - x }^2=\dfrac{1-\sigma}{\lambda}\norm{y-x}^2.
    \]
		To end the proof, combine the two above
    inequalities.
\end{proof}
Note that~\eqref{eq:is} allows errors
in both the inclusion and the equation in~\eqref{eq:ps}.
Indeed, since $\partial f(y)\subset \partial_\varepsilon f(y)$
it is easy to see that every triple $(\lambda,y,v)$
satisfying~\eqref{eq:ps} also satisfies~\eqref{eq:is} with $\varepsilon=0$.
Moreover, if $\sigma=0$ in~\eqref{eq:is}
then we have that $(\lambda,y,v)$ satisfies~\eqref{eq:ps}.

Motivated by the above results, we will now state our method
which uses approximate solutions of~\eqref{eq:p2}, in the sense 
of Proposition~\ref{pr:ie}.

\medskip

\noindent
\fbox{
\begin{minipage}[h]{5.9 in}
{\bf Algorithm 1:} A Large-step PP method for convex optimization
\begin{itemize}
\item[(0)] Let $x_0\in \mbox{dom}(f)$, $\sigma\in [0,1[$, $\theta>0$ be given,
 and set $k=1$;
\item[(1)] choose $\lambda_k>0$, and find $x_k,v_k\in \HH$, $\varepsilon_k\geq 0$
	  such that
       \begin{align}
  \label{eq:hy.inc}
  &v_k \in\partial_{\varepsilon_k} f(x_k),\\
  \label{eq:hy.res}
  &\norm{\lambda_kv_k+x_k-x_{k-1}}^2+2\lambda_k\varepsilon_k\leq
  \sigma^2\norm{x_k-x_{k-1}}^2,
  \\
  \label{eq:bigstep}
  &\lambda_k\norm{x_k-x_{k-1}}\geq \theta \mbox{ or } v_k=0;
\end{align}
\item[(2)] if $v_k=0$ then {\sc STOP} and output $x_k$; otherwise
let $k\leftarrow k+1$ and go to step 1.
\end{itemize}
\noindent
{\bf end}
\end{minipage}
}

\bigskip

We now make some comments about Algorithm 1.
First, the error tolerance \eqref{eq:hy.inc}-\eqref{eq:hy.res} is a
particular case of the relative error tolerance for the
HPE/Projection method introduced in~\cite{So-Sv1,So-Sv2}, but here we are not performing an
extragradient step, while the inequality in~\eqref{eq:bigstep} was used/introduced by Monteiro and Svaiter
in~\cite{MS1,MS2}. Second, as in the recent literature on the HPE method, we assume
that the vectors and scalars in step (1) are given by a black-box. Concrete instances of such a black-box
would depend on the particular implementation of the method. We refer the reader to the next section, where it
is shown that (in the smooth case) a single Newton step for the proximal subproblem provides scalars and vectors
satisfying all the conditions of step (1).

From now on in this section, $\{x_k\}$, $\{v_k\}$,
$\{\varepsilon_k\}$ and $\{\lambda_k\}$  are sequences
generated by Algorithm 1.  These sequences may be finite or infinite.
The provision for $v_k=0$ is in~\eqref{eq:bigstep} because, in this case, $x_{k-1}$ is
already a minimizer of $f$, as proved in the sequel.
%

\begin{proposition}
  \label{pr:bas}
	For $x_0\in \HH$, assume that iteration $k\geq 1$ of Algorithm 1 is reached (so that $\lambda_k$, $x_k$, $v_k$ and
  $\varepsilon_k$ are generated). Then, the following statements hold:
	\begin{enumerate}
  \item[(a)]\label{it:b11}
    $ f(x')\geq f(x_k)+\inner{v_k}{x'-x_k}
    -\varepsilon_k \qquad \forall x'\in \HH$;
  \item[(b)] \label{it:b22}
    $ \displaystyle f(x_{k-1})\geq f(x_k)+
    \frac{\lambda_k}{2} \norm{v_k}^2+\frac{1-\sigma^2}{2\lambda_k}
    \norm{x_k-x_{k-1}}^2\geq f(x_k)$;
  \item[(c)]\label{it:b33}
    $ (1+\sigma)\norm{x_k-x_{k-1}} \geq \norm{\lambda_k v_k}
    \geq (1-\sigma)\norm{x_k-x_{k-1}}$;
  \item[(d)]\label{it:b44} $\displaystyle
    \varepsilon_k\leq \frac{\sigma^2}{2(1-\sigma)}
    \norm{v_k}\,\norm{x_k-x_{k-1}}$;
  \end{enumerate}
and
\begin{align}
\label{eq:pr:bas}
     \frac{\lambda_k}{2} \norm{v_k}^2+\frac{1-\sigma^2}{2\lambda_k}
    \norm{x_k-x_{k-1}}^2 \geq
 \max \left\{
      \norm{v_k}^{3/2}\sqrt{\theta(1-\sigma)},
      \frac{1-\sigma}{\lambda_k} \norm{x_k-x_{k-1}}^2 \right\}.
  \end{align}
   \item[(e)]\label{it:b55} Suppose $\inf f > - \infty$. Then $ \sum \frac{1}{{\lambda_k}^{3}} < + \infty $; as a consequence, if the sequences $\{\lambda_k\}$, $\{x_k\}$ etc. are infinite, then $\lambda_k \to + \infty$ as $k \to \infty$.

\end{proposition}
\begin{proof}
Items \textit{ (a)}, \textit{(b)}, \textit{(c)}, and \textit{(d)} follow directly from Proposition~\ref{pr:ie}
and Algorithm 1's definition. To prove \textit{(e)}, first notice that \textit{(b)} implies, for any $j\geq 1$
\[
f(x_{j-1})\geq f(x_j) +   \frac{1-\sigma^2}{2\lambda_j} \norm{x_j-x_{j-1}}^2.
\]
Summing this inequality from $j=1$ to $k$, we obtain
\[
f(x_0)\geq f(x_k) +   \frac{1-\sigma^2}{2} \sum_{j=1}^k \frac{\norm{x_j-x_{j-1}}^2}{\lambda_j}.
\]
Note that, in order Algorithm 1 to be defined, we need to take $x_0 \in \mbox{dom} f$, i.e., $f(x_0) <+\infty$. Since, by assumption,  $\inf f > - \infty$, and $\sigma <1$, we deduce that
\begin{equation}\label{eq:sum1}
 \sum_{k} \frac{\norm{x_k-x_{k-1}}^2}{\lambda_k} < + \infty.
\end{equation}
On the other hand, by definition of Algorithm 1, \eqref{eq:bigstep}, we have 
 $\lambda_k\norm{x_k-x_{k-1}}\geq \theta $. Equivalently,  $\norm{x_k-x_{k-1}}^2\geq \frac{\theta^2}{{\lambda_k}^2} $. Combining this inequality with \eqref{eq:sum1}, and $\theta >0$, we obtain
 \begin{equation}\label{eq:sum2}
 \sum_{k} \frac{1}{{\lambda_k}^3} < + \infty.
\end{equation}
\end{proof}
Suppose now that Algorithm 1 generates infinite sequences.
Any convergence result valid under this
assumption is valid in the general case, with the provision ``or a
solution is reached in a finite number of iterations''.
We are ready to analyze the (global) rate of convergence and the iteration complexity of Algorithm 1.
To this end, let $\mathcal{D}_0$ be the diameter of the level set $[f\leq f(x_0)]$, that is,
\begin{align}
  \label{eq:diameter}
 \mathcal{D}_0	=&\sup\{ \norm{x-y}\;|\; \max\{f(x),f(y)\}\leq f(x_0)\}.
\end{align}

\begin{theorem}
  \label{th:gen}
  Assume that $ \mathcal{D}_0<\infty$, let $\bar x$ be
	a solution of~\eqref{eq:p2} and define
  \begin{equation}
	\label{eq:def.dh}
    \widehat D=\mathcal{D}_0
    \left[ 1+\frac{\sigma^2}{2(1-\sigma)}
    \right],\;\;\;
    \kappa=\sqrt{\frac{\theta(1-\sigma)}{\widehat D^3}}.
  \end{equation}
  Then, the following statements hold for every $k\geq 1$:
  \begin{enumerate}
    \item[(a)]
      \label{it:g.1}
      $\norm{v_k}\widehat D\geq
      f(x_k)-f(\bar x)$;
    \item[(b)]
      \label{it:g.2}
      $f(x_k) \leq f(x_{k-1})-\dfrac{2\kappa(f(x_{k-1})-f(\bar x))^{3/2}}
      {2+3\kappa(f(x_{k-1})-f(\bar x))^{1/2}}$;
      \item[(c)]
      \label{it:g.3}
 $\displaystyle
  f(x_k)-f(\bar x)\leq
  \frac{f(x_0)-f(\bar x)}{\left[1+k
    \dfrac{\kappa\sqrt{f(x_0)-f(\bar x)}}{2+3\kappa\sqrt{f(x_0)-f(\bar x)}}
    \right]^2}={\cal O}(1/k^2).
  $
     \end{enumerate}
   Moreover,
  for each $k\geq 2$ even, there exists $j\in\{k/2+1,\dots,k\}$ such that
  \begin{equation}
    \norm{ v_j}\leq
    \dfrac{4}{\sqrt[3]{\theta(1-\sigma)}}
  \left[\frac{f(x_0)-f(\bar x)}{k
  \left[2+k\dfrac{\kappa\sqrt{f(x_0)-f(\bar x)}}
{2+3\kappa\sqrt{f(x_0)-f(\bar x)}}\right]^2}
  \right]^{2/3}
  =\mathcal{O}(1/k^2)
    \label{eq:th.nabla}
  \end{equation}
  and
  \begin{align}
	\label{eq:th.nabla2}
  \varepsilon_j\leq \frac{4\sigma^2}{(1-\sigma)}
  \frac{f(x_0)-f(\bar x)}
  {k\left[2+k\dfrac{\kappa\sqrt{f(x_0)-f(\bar x)}}
{2+3\kappa\sqrt{f(x_0)-f(\bar x)}}\right]^2}
  =\mathcal{O}(1/k^3).
  \end{align}
\end{theorem}

\begin{proof}
\textit{(a)} In view of Proposition~\ref{pr:bas}\textit{(b)}
and the fact that $\bar x$ is a solution of~\eqref{eq:p2}
we have $\max\{f(x_k),f(\bar x)\}\leq f(x_0)$ for all $k\geq 0$.
As a consequence of the latter inequality and~\eqref{eq:diameter}
we find
\begin{align}
\label{eq:51}
  \max\{\norm{\bar x-x_{k-1}},\norm{x_k-x_{k-1}}\}\leq \mathcal{D}_0\quad \forall k\geq 1.
\end{align}
 Using Proposition~\ref{pr:bas}\textit{(a)} with $x'=\bar x$,
 Proposition~\ref{pr:bas}\textit{(d)} and the Cauchy-Schwarz
inequality we conclude that
	\begin{align*}
    f(x_k)-f(\bar x) \leq \inner{v_k}{x_k-\bar x}+\varepsilon_k
    & \leq
    \norm{v_k}\norm{x_k-\bar x}+\frac{\sigma^2}{2(1-\sigma)}
    \norm{v_k}\,\norm{x_k-x_{k-1}}\quad \forall k\geq 1,
  \end{align*}
	which in turn combined with~\eqref{eq:51}
	and the definition of $\widehat D$ in~\eqref{eq:def.dh} proves \textit{(a)}.

	\textit{(b)} By Proposition~\ref{pr:bas}\textit{(b)},~\eqref{eq:pr:bas}, the above item \textit{(a)},
	and the definition of $\kappa$ in~\eqref{eq:def.dh}
	we have for all $k\geq 1$:
  \begin{align}
	\label{eq:54}
  f(x_{k-1})-f(\bar x)
	\nonumber
  &\geq f(x_k)-f(\bar x)+\norm{v_k}^{3/2}\sqrt{\theta(1-\sigma)}\\
  &\geq f(x_k)-f(\bar x)+\kappa(f(x_k)-f(\bar x))^{3/2}.
  \end{align}
	Using the latter inequality and Lemma~\ref{lm:aux1}
	(for each $k\geq 1$) with $b=f(x_k)-f(\bar x)$,
	$a=f(x_{k-1})-f(\bar x)$ and $c=\kappa$ we obtain
	\begin{align}
	\label{eq:52}
	 f(x_k)-f(\bar x)\leq f(x_{k-1})-f(\bar x)-\dfrac{\kappa(f(x_{k-1})-f(\bar x))^{3/2}}{1+(3\kappa/2)(f(x_{k-1})-f(\bar x))^{1/2}}\quad \forall k\geq 1,
	\end{align}
	which in turn proves \textit{(b)}.

	\textit{(c)} Defining $a_k :=f(x_k)-f(\bar x)$, $\tau:=2\kappa/(2+3\kappa a_0)$, and using the second inequality in           Proposition~\ref{pr:bas}\textit{(b)}
	and~\eqref{eq:52}, we conclude that
	\[
	 a_k\leq a_{k-1}-\tau a_{k-1}^{3/2}\quad \forall k\geq 1,
	\]
	which leads to \textit{(c)}, by direct application of  Lemma~\ref{lm:tch2} (see Appendix).

	To prove the last statement of the theorem, assume that
	$k\geq 2$ is even. Using the first inequality in Proposition~\ref{pr:bas}\textit{(b)},
	we obtain
	 \begin{align}
	\label{eq:53}
   f(x_{k/2})-f(x_{k})=\sum_{i=k/2+1}^k f(x_{i-1})-f(x_{i})
   \geq  \sum_{i=k/2+1}^k
	\frac{\lambda_i}{2} \norm{v_i}^2+\frac{1-\sigma^2}{2\lambda_i}
	\norm{x_i-x_{i-1}}^2.
   \end{align}
	Taking $j\in \{k/2+1,\dots,k\}$ which minimizes the general term
  in the second sum of the latter inequality, and using the fact that
	$\bar x$ is a solution of~\eqref{eq:p2},
	we have
   \begin{align*}
	f(x_{k/2})-f(\bar x)\geq
       \dfrac{k}{2}\left[
	\frac{\lambda_j}{2} \norm{v_j}^2+\frac{1-\sigma^2}{2\lambda_j}
	\norm{x_j-x_{j-1}}^2\right],
   \end{align*}
	which, in turn, combined with~\eqref{eq:pr:bas} and~\eqref{eq:hy.res} gives
	 \begin{align*}
	 \frac{f(x_{k/2})-f(\bar x)}{k/2}&\geq \max\left\{\norm{v_j}^{3/2}\sqrt{\theta(1-\sigma)},
				\frac{1-\sigma}{\lambda_j}\norm{x_j-x_{j-1}}^2\right\}\\
				&\geq \max\left\{\norm{v_j}^{3/2}\sqrt{\theta(1-\sigma)},
				\frac{2(1-\sigma)}{\sigma^2}\varepsilon_j\right\}.
   \end{align*}
	Combining the latter inequality with \textit{(c)}, and using some trivial algebraic manipulations, we obtain~\eqref{eq:th.nabla}
	and~\eqref{eq:th.nabla2}, which finishes the proof of the theorem.
   \end{proof}

We now prove that if $\varepsilon_k=0$ in Algorithm 1, then better complexity
bounds can be obtained.

\begin{theorem}
  \label{th:e0}
  Assume that $\mathcal{D}_0<\infty$, and $\varepsilon_k=0$ for all $k\geq 1$.
	Let $\bar x$ be a solution of~\eqref{eq:p2} and define
	\begin{align*}
  \kappa_0=\sqrt{\frac{\theta(1-\sigma)}{\mathcal{D}_0^3}}.
  \end{align*}
	Then, the following statements hold for all $k\geq 1$:
  \begin{enumerate}
    \item[(a)]
      \label{it:e0.1}
      $\norm{v_k} \mathcal{D}_0\geq f(x_k)-f(\bar x)$;
    \item[(b)]
      \label{it:e0.2}
      $ f(x_k) \leq f(x_{k-1})-\dfrac{2\kappa_0(f(x_{k-1})-f(\bar x))^{3/2}}
      {2+3\kappa_0(f(x_{k-1})-f(\bar x))^{1/2}}$;
    \item[(c)]
      \label{it:e0.3}
 $\displaystyle
    f(x_k)-f(\bar x)\leq
  \frac{f(x_0)-f(\bar x)}{\left[1+k
    \dfrac{\kappa_0\sqrt{f(x_0)-f(\bar x)}}{2+3\kappa_0\sqrt{f(x_0)-f(\bar x)}}
    \right]^2}={\cal O}(1/k^2).
  $
   \end{enumerate}
   Moreover,
  for each $k\geq 2$ even, there exists $j\in\{k/2+1,\dots,k\}$ such that
  \begin{equation}
    \norm{ v_j}\leq
    \dfrac{4}{\sqrt[3]{\theta(1-\sigma)}}
  \left[\frac{f(x_0)-f(\bar x)}{k
  \left[2+k\dfrac{\kappa_0\sqrt{f(x_0)-f(\bar x)}}
{2+3\kappa_0 \sqrt{f(x_0)-f(\bar x)}}\right]^2}
  \right]^{2/3}
  =\mathcal{O}(1/k^2).
    \label{eq:th2.nabla}
  \end{equation}
\end{theorem}
\begin{proof}
By the same reasoning as in the proof of Theorem~\ref{th:gen}\textit{(a)}
we obtain~\eqref{eq:51}
and
\begin{align*}
 f(x_k)-f(\bar x)\leq \norm{v_k}\norm{x_k-\bar x}\leq \norm{v_k}\mathcal{D}_0 \quad \forall k\geq 1,
\end{align*}
which in turn proves \textit{(a)}. Using \textit{(a)}, the definition of $\kappa_0$,
and the same reasoning as in the proof of Theorem~\ref{th:gen}\textit{(b)}
we deduce that~\eqref{eq:54} holds with $\kappa_0$
in the place of $\kappa$.
The rest of the proof is analogous to that of Theorem~\ref{th:gen}.
\end{proof}

In the next corollary, we prove that Algorithm 1 is able to find approximate solutions
of the problem~\eqref{eq:p2} in at most ${\cal O}(1/\sqrt{\varepsilon})$
iterations.

\begin{corollary}
  \label{cr:cp.g}
  Assume that all the assumptions of Theorem~\ref{th:e0} hold, and let
	$\varepsilon>0$ be a given tolerance.
  Define
  \begin{align}
	\label{eq:def.k.j}
    K=\dfrac{2+3\kappa_0\sqrt{f(x_0)-f(\bar x)}}{\kappa_0\sqrt{\varepsilon}},
    \quad
    J=\dfrac{2}{(\theta(1-\sigma))^{1/6}}\;
    \dfrac{(2+3\kappa_0\sqrt{f(x_0)-f(\bar x)})^{2/3}}
    {\kappa_0^{1/3}\sqrt{\varepsilon}}.
  \end{align}
  Then, the following statements hold:
  \begin{enumerate}
  \item[(a)] for any $k\geq K$, $f(x_k)-f(\bar x)\leq\varepsilon$;
  \item[(b)] there exists $j\leq 2\left\lceil J\right\rceil$ such that $\norm{v_j}\leq\varepsilon$.
  \end{enumerate}
\end{corollary}

\begin{proof}
 The proof of \textit{(a)} and \textit{(b)} follows trivially from Theorem~\ref{th:e0}\textit{(c)}
and~\eqref{eq:th2.nabla}, respectively, and from~\eqref{eq:def.k.j}.
 \end{proof}

\section{An ${\cal O}(1/\sqrt{\varepsilon})$ proximal-Newton method for
  smooth convex optimization}
\label{sec:csm}
In this section, we consider
a proximal-Newton method for solving the convex optimization problem
\begin{align}
  \label{eq:p3}
  \mbox{minimize}\,f(x)\quad \mbox{s.t.}\quad  x\in \HH,
\end{align}
where $f:\HH\to \R$, and the following assumptions are made:
\begin{description}
 \item[AS1)] $f$ is convex and twice continuously differentiable;
 \item[AS2)] the Hessian of $f$ is $L$-Lipschitz continuous, that is, there exists $L>0$
such that
  \begin{align*}
    \norm{\nabla^2f(x)-\nabla^2f(y)}\leq L\norm{x-y}\quad \forall x,y\in \HH
  \end{align*}
  where, at the left hand-side,  the operator norm is
  induced by the Hilbert norm of $\HH$;
\item[AS3)] there exists a solution of~\eqref{eq:p3}.
\end{description}
{\bf Remark.}
It follows from Assumptions {\bf AS1} and {\bf AS2} that $\nabla^2f(x)$
exists and is positive semidefinite (psd) for all $x\in \HH$, while it
follows from Assumption {\bf AS2} that
\begin{align}
  \label{eq:qerror}
  \norm{\nabla f(y)-\nabla f(x)-\nabla^2f(x)(y-x)}\leq
\dfrac{L}{2}\norm{y-x}^2\quad \forall x,y\in \HH.
\end{align}

Using assumption \textbf{AS1}, we have that an exact proximal point iteration at $x\in \HH$, with stepsize $\lambda>0$,
consists in finding $y\in \HH$ such that
\begin{align}
  \label{eq:proxgrad}
 \lambda \nabla f(y)+y-x=0 \quad \mbox{(cf.~\eqref{eq:ps})}.
\end{align}
The basic idea of our method is to perform a single Newton iteration
for the above equation from the current iterate $x$, i.e., in computing the (unique) solution $y$ of
the linear system
\begin{align*}
    \lambda (\nabla f(x)+\nabla^2 f(x)(y-x))+y-x=0 ,
\end{align*}
and defining the new iterate as  such $y$.
We will show that,
due to~\eqref{eq:qerror},
it is possible to choose $\lambda$ 
so that:
a) condition~\eqref{eq:is} is satisfied with $\varepsilon=0$ and $v=\nabla f(y)$;
b) a large-step type condition (see \eqref{eq:bigstep}) is satisfied
for $\lambda$, $x$ and $y$.
First we show that Newton step 
is well defined and find bounds for its norm.

\begin{lemma}
  \label{lm:ns0}
  For any $x\in \HH$, if  $\lambda>0$ then $\lambda \nabla^2f(x)+I$
  is nonsingular and
  \begin{align}
    \label{eq:lsis2}
    \dfrac{\lambda\norm{\nabla f(x)}}{\lambda\norm{\nabla^2f(x)}+1}
    \leq \norm{(\lambda \nabla^2f(x)+I)^{-1}\lambda\nabla f(x)}\leq\lambda\norm{\nabla f(x)}.
  \end{align}
\end{lemma}

\begin{proof}
  Non-singularity of $\lambda\nabla^2f(x)+I$, as well as the inequalities
  in~\eqref{eq:lsis2}, are due to the facts that $\lambda>0$, $\nabla^2
  f(x)$ is psd (see the remark after the Assumption \textbf{AS3}), and the
  definition of operator's norm.
\end{proof}

The next result provides {\it a priori} bounds for the (relative)
residual in \eqref{eq:proxgrad} after a Newton iteration from $x$ for
this equation.

\begin{lemma}
  \label{lm:ns1}
  For any $x\in \HH$, if $\lambda>0$, $\sigma>0$, and
 \begin{align}
   \label{eq:ynew}
   y=x-(\lambda\nabla^2f(x)+I)^{-1}\lambda\nabla f(x),\qquad
    \lambda\norm{(\lambda\nabla^2f(x)+I)^{-1}\lambda \nabla f(x)}
    \leq \dfrac{2\sigma}{L},
  \end{align}
  then
 $\norm{\lambda \nabla f(y)+y-x}\leq \sigma\norm{y-x}$.
\end{lemma}

\begin{proof}
  It follows from
  \eqref{eq:ynew} that
  \begin{align*}
    \lambda\nabla f(y)+y-x&=\lambda\nabla f(y)
    -\lambda[\nabla f(x)+\nabla^2f(x)(y-x)],
    \qquad \lambda\norm{y-x}\leq\dfrac{2\sigma}{L}.
  \end{align*}
  Therefore
  \begin{align*}
    \norm{\lambda\nabla f(y)+y-x}&=\lambda\norm{\nabla f(y)-
      \nabla f(x)-\nabla^2f(x)(y-x)}\leq \dfrac{\lambda L}{2}\norm{y-x}^2
    \leq \sigma\norm{y-x},
  \end{align*}
  where the first inequality follows from~\eqref{eq:qerror}.
  %
\end{proof}

\begin{lemma}
  \label{lm:ns2}
  For any $x\in \HH$, and $0<\sigma_\ell<\sigma_u<+\infty$, if $\nabla f(x) \neq 0$ then
    the set of all scalars $\lambda\in]0,+\infty[$ satisfying
    \begin{align}
      \label{eq:55}
      \dfrac{2\sigma_\ell }{L}\leq \lambda
      \norm{(\lambda\nabla^2f(x)+I)^{-1}\lambda\nabla f(x)}
      \leq \frac{2\sigma_u}{L}
		\end{align}
		is a (nonempty) closed interval $[\lambda_\ell,\lambda_u]\subset ]0,+\infty[$,
                \begin{align}
                  \label{eq:bds}
&                  \sqrt{\dfrac{2\sigma_\ell/L}{\norm{\nabla
                        f(x)}}}\leq\lambda_\ell,\;\;
\lambda_u \leq \dfrac{\dfrac{\norm{\nabla^2f(x)}\sigma_u}L
  +\sqrt{
                      \left(\dfrac{\norm{\nabla^2f(x)}\sigma_u}L\right)^2+
                    \dfrac{\norm{\nabla f(x)}2\sigma_u}L}}{\norm{\nabla f(x)}}
                 \end{align}
 and 
 $\lambda_u/\lambda_\ell \geq \sqrt{\sigma_u/\sigma_\ell}$.
\end{lemma}

\begin{proof}
  Assume that $\nabla f(x)$ is nonzero.  Define the operator $A:\HH\to
  \HH$ by $A(y)=\nabla f(x)+\nabla^2f(x)(y-x)$. Since $\nabla^2 f(x)$ is
  psd, it follows that the affine linear operator $A$
  is maximal monotone. It can be easily checked that, in this setting,
  \begin{align}
    \label{eq:phi.ns2}
    J_{\lambda}^A(x)=x-(\lambda \nabla^2f(x)+I)^{-1}\lambda\nabla f(x),
    \quad
    \varphi(\lambda,x)=\lambda
    \norm{(\lambda\nabla^2f(x)+I)^{-1}\lambda\nabla f(x)},    
  \end{align}
 (see~\eqref{eq:psi} and the paragraph below~\eqref{eq:psi}
  to recall the notation).
Hence, using Proposition~\ref{pr:phi-property} we conclude
that there exists $0<\lambda_\ell<\lambda_u<\infty$
such that
\begin{align}
  \label{eq:phi.u,v}
\varphi(\lambda_\ell,x)=\dfrac{2\sigma_\ell}{L},\quad
\varphi(\lambda_u,x)=
\dfrac{2\sigma_u}{L},
\end{align}
%
%
and the set of all scalars satisfying~\eqref{eq:55} is the closed
interval $[\lambda_\ell,\lambda_u]\subset]0,+\infty[$.  It follows from the
second inequality in~\eqref{eq:psi1} and the above (implicit)
definitions of $\lambda_\ell$ and $\lambda_u$ that
\[
\dfrac{2\sigma_u}{L}=\varphi(\lambda_u,x)\leq
\left(\dfrac{\lambda_u}{\lambda_\ell}\right)^2\varphi(\lambda_\ell,x)=
\left(\dfrac{\lambda_u}{\lambda_\ell}\right)^2\dfrac{2\sigma_\ell}{L}
\]
which trivially implies
that $\lambda_u/\lambda_\ell\geq\sqrt{\sigma_u/\sigma_\ell}$.
To prove the two inequalities in~\eqref{eq:bds}, first observe that, in view of
the expression \eqref{eq:phi.ns2} for $\varphi(\lambda,x)$, and Lemma~\ref{lm:ns0},
 we have
\begin{align*}
  \dfrac{\lambda^2\norm{\nabla f(x)}}{\lambda\norm{\nabla^2f(x)}+1}\leq
  \varphi(\lambda,x)\leq\lambda^2\norm{\nabla f(x)}.
\end{align*}
Then, evaluate these inequalities for $\lambda=\lambda_\ell$, $\lambda=\lambda_u$,
and use the above implicit expression \eqref{eq:phi.u,v}  for $\lambda_\ell$ and $\lambda_u$.
\end{proof}

Motivated by the above results, we propose the following algorithm
for solving~\eqref{eq:p3}. This algorithm is the main object of study in this section.
We will prove that, for a given tolerance $\varepsilon>0$, it is able to find approximate solutions of~\eqref{eq:p3} in at most
${\cal O}(1/\sqrt{\varepsilon})$ iterations, i.e., it has the same complexity as the cubic regularization of the Newton method
proposed and studied in~\cite{NP}.

\bigskip
\noindent
\fbox{
\begin{minipage}[h]{5.9 in}
{\bf Algorithm 2:} A proximal-Newton method for convex optimization
\begin{itemize}
\item[(0)] Let $x_0\in \HH$, $0<\sigma_\ell<\sigma_u<1$ be given,
 and set $k=1$;
\item[(1)] if $\nabla f(x_{k-1})=0$ then {\bf stop}. Otherwise,
 compute $\lambda_k>0$ such that
	\begin{align}
	\label{eq:ls.n}
    \dfrac{2\sigma_\ell }{L}\leq
    \lambda_k\norm{(I+\lambda_k\nabla^2 f(x_{k-1}))^{-1}\lambda_k\nabla f(x_{k-1})}
    \leq \dfrac{2\sigma_u}{L};
  \end{align}
\item[(2)] set $\displaystyle
  x_k=x_{k-1}-(I+\lambda_k\nabla^2f(x_{k-1}))^{-1}\lambda_k\nabla
  f(x_{k-1})$; 
\item[(3)] set $k\leftarrow k+1$ and go to step 1.
\end{itemize}
\noindent
{\bf end}
\end{minipage}
}
\bigskip

\noindent
{\bf Remark.}
We note that, for a given $\lambda_k>0$, iterate
$x_k$, defined in step (2) of Algorithm 2, is the solution of the
quadratic problem
\[
\min_{x\in \HH}\, f(x_{k-1})+\inner{\nabla
  f(x_{k-1})}{x-x_{k-1}}+\dfrac{1}{2}\inner{x-x_{k-1}}{\nabla^2
  f(x_{k-1})(x-x_{k-1})}+\dfrac{1}{2\lambda_k}\norm{x-x_{k-1}}^2.
\]
Hence, our method is based on classical quadratic regularizations of
quadratic local models for $f$, combined with
a large-step type condition.

\medskip

At iteration $k$, we must find $\lambda_k\in[\lambda_\ell,\lambda_u]$, where
\begin{align*}
  \lambda_\ell=\Lambda_{2\sigma_\ell/L}(x_{k-1}),\quad
  \lambda_u=\Lambda_{2\sigma_u/L}(x_{k-1}).
\end{align*}
Lemma~\ref{lm:ns2} provides a lower and an upper bound for
$\lambda_\ell$ and $\lambda_u$ respectively, and guarantees that the length
of the interval $[\log\lambda_\ell,\log\lambda_u]$ is no smaller than 
$\log(\sigma_u/\sigma_\ell)/2$. A binary search in $\log\lambda$ may be
used for finding $\lambda_k$. The complexity of such a procedure
was analysed in~\cite{MS1,MS2}, in the context of the HPE method.
The possible improvement of this procedure is a subject of future research.

\begin{proposition}
  \label{pr:alg.1}
  For $x_0\in \HH$ and $0<\sigma_\ell<\sigma_u<1$, consider the
  sequences $\{\lambda_k\}$ and $\{x_k\}$
  generated by Algorithm 2 and define
	\begin{align}
	\label{eq:61}
	 \sigma=\sigma_u,\quad \theta=2\sigma_\ell/L,\quad v_k=\nabla f(x_k), \quad \varepsilon_k=0\quad \forall k\geq 1.
	\end{align}
	Then, the following statements hold for every $k\geq 1$:
	\begin{itemize}
	\item[(a)] $v_k\in \partial_{\varepsilon_k} f(x_k)$,\quad $\norm{\lambda_k v_k+x_k-x_{k-1}}\leq \sigma \norm{x_k-x_{k-1}}$;
	\item[(b)] $\lambda_k\norm{x_k-x_{k-1}}\geq \theta$;
	\item[(c)] $\lambda_k\geq \sqrt{\sigma_\ell/(1+\sigma_u)\sigma_u}\lambda_{k-1}$;
	\item[(d)] $v_k$ is nonzero whenever $v_0$ is nonzero.
	\end{itemize}
	As a consequence,  Algorithm 2 is a special instance of  Algorithm 1,
	with $\sigma$, $\theta$ and the sequences $\{v_k\}$ and $\{\varepsilon_k\}$
	given by~\eqref{eq:61}.
\end{proposition}
\begin{proof}
\textit{(a)} First note that the inclusion in \textit{(a)} follows trivially from the definition of $v_k$
and $\varepsilon_k$ in~\eqref{eq:61}. Moreover, using the definitions of $\sigma$ and $v_k$ in~\eqref{eq:61}, the second inequality in~\eqref{eq:ls.n},
the definition of $x_k$ in step 2 of Algorithm 2,
and Lemma~\ref{lm:ns1} with $\lambda=\lambda_k$, $y=x_k$ and $x=x_{k-1}$
we obtain
\begin{align*}
\norm{\lambda_k v_k+x_k-x_{k-1}}&=\norm{\lambda_k \nabla f(x_k)+x_k-x_{k-1}}
\leq \sigma\norm{x_k-x_{k-1}},
\end{align*}
which concludes the proof of \textit{(a)}.

\textit{(b)} The statement in \textit{(b)} follows easily from the definition of $x_k$ and $\theta$ in step 2
of Algorithm 2 and~\eqref{eq:61}, respectively, and the first inequality in~\eqref{eq:ls.n}.

\textit{(c)} Using Algorithm 2's definition, 
item~\textit{(a)}, and Lemma~\ref{lm:ns0} with $\lambda=\lambda_k$,
$x=x_{k-1}$ we have, for all $ k\geq 1$
\begin{align}
\label{eq:71}
\lambda_k \norm{ \nabla f(x_k)}\leq (1+\sigma_u)\norm{(\lambda_k\nabla^2f(x_{k-1})+I)^{-1}\lambda_k\nabla f(x_{k-1})}\leq
(1+\sigma_u)\lambda_k\norm{\nabla f(x_{k-1})}.
\end{align}
Set $s_k = -(\lambda_k\nabla^2f(x_{k-1})+I)^{-1}\lambda_k\nabla f(x_{k-1}) $.
Note now that~\eqref{eq:ls.n} and the definition of $s_k$
imply that $2\sigma_\ell/L\leq \norm{\lambda_js_j}\leq 2\sigma_u/L$
for all $j=1,\cdots, k$. Direct use of the latter inequalities for $j=k-1$ and $j=k$, and the multiplication of the second inequality in the latter displayed equation by $\lambda_{k-1}^2\lambda_k$ yield
\begin{align*}
 \lambda_{k-1}^2(2\sigma_\ell)/L\leq \lambda_{k-1}^2\lambda_k^2\norm{\nabla f(x_{k-1})}&=
 \lambda_k^2\lambda_{k-1}\norm{\lambda_{k-1}\nabla f(x_{k-1})}\\
 &\leq (1+\sigma_u)\lambda_k^2\norm{\lambda_{k-1}s_{k-1}}\\
 &\leq (1+\sigma_u)\lambda_k^2(2\sigma_u)/L,
\end{align*}
and, hence, the inequality in \textit{(c)}.

\textit{(d)} To prove this statement observe that if $\nabla f(x_{k-1}) \neq 0$ then
$x_k\neq x_{k-1}$, and use item \textit{(a)}, the second inequality in item \textit{(c)} of
Proposition~\ref{pr:ie}, and induction in $k$.
\end{proof}

Now we make an additional assumption in order to derive
complexity estimates for the sequence generated
by Algorithm 2.
\begin{description}
\item[AS4)] The level set $\{x\in \HH\:|\; f(x)\leq f(x_0)\}$ is
  bounded, and $\mathcal{D}_0$ is its diameter, that is,
  \begin{align*}
    \mathcal{D}_0=\sup\{\norm{y-x}\;|\; \max\{f(x),f(y)\}\leq f(x_0)\}<\infty.
  \end{align*}
\end{description}

\begin{theorem}
  \label{th:main}
	Assume that assumptions AS1, AS2, AS3, AS4 hold, and consider the sequence $\{x_k\}$
	generated by Algorithm 2. Let $\bar x$ be a solution of~\eqref{eq:p3}
	and, for any given tolerance $\varepsilon>0$ define
	\begin{align*}
	 \kappa_0=\sqrt{\dfrac{2\sigma_\ell(1-\sigma_u)}{LD_0^3}},
	\quad K=\dfrac{2+3\kappa_0\sqrt{f(x_0)-f(\bar x)}}{\kappa_0\sqrt{\varepsilon}},
	 \ \ J=\dfrac{2L^{1/6}\left(2+3\kappa_0\sqrt{f(x_0)-f(\bar x)}\right)^{2/3}}
    {\left[2\sigma_\ell(1-\sigma_u)\right]^{1/6}\kappa_0^{1/3}\sqrt{\varepsilon}}
	\end{align*}
	Then, the following statements hold for every $k\geq 1$:
  \begin{enumerate}
	\item[(a)] for any $k\geq K$, $f(x_k)- f(\bar x)\leq\varepsilon$;
  \item[(b)] there exists $j\leq 2\left\lceil J\right\rceil$ such that $\norm{\nabla f(x_j)}\leq\varepsilon$.
  \end{enumerate}
\end{theorem}

\begin{proof}
  The proof follows from the last statement of
  Proposition~\ref{pr:alg.1} and Corollary~\ref{cr:cp.g}.
\end{proof}
In practical implementations of Algorithm 2, as in other  Newton methods,
the main iteration is divided into two steps: the computation of a Newton step $s_k$,
\[
s_k=-(\lambda_k\nabla^2f(x_{k-1})+I)^{-1}\lambda_k\nabla f(x_{k-1}),
\]
and the update $x_k=x_{k-1}+s_k$. As in other Newton methods, step $s_k$
is not to be computed using the inverse of
$\lambda_k\nabla^2f(x_{k-1})+I$. Instead, the linear system
$$
(\nabla^2f(x_k)+\mu_kI)s_k = -\nabla f(x_{k-1}), \qquad \mu_k=1/\lambda_k
$$
is solved via a Hessenberg factorization (followed by a Choleski factorization), a Cholesky factorization, or a conjugate gradient method.
Some reasons for choosing a Hessenberg factorization are discussed in~\cite{MS1}.
For large and dense linear systems, conjugate gradient is the method of choice, and
it is used as an iterative procedure.
In these cases, the linear system is not solved (exactly). Even for
Hessenberg and Cholesky factorization, ill-conditioned linear systems are
inexactly solved with a non-negligible error.

Since $\lambda_k\to\infty$, $\mu_k\to 0$ and, in spite of the
regularizing term $\mu_kI$, ill-conditioned systems may occur. For these
reasons, it may be interesting to consider a variant of Algorithm 2
where an ``inexact'' Newton step is used, see ~\cite{MS2} for the
development of this method in the context of the HPE method.
%

\subsection{Quadratic convergence in the regular case}
\label{sec:qc}

In this section, we will analyze Algorithm 2 under the assumption:

\smallskip

\noindent
\textbf{AS3r)} there exists a unique $x^*$ solution of \eqref{eq:p3}, and
$\nabla^2 f(x^*)$ is non-singular.

\begin{theorem}\label{quad-conv-thm}
 Let us make assumptions AS1, AS2, and AS3r. Then, the sequence $\{x_k\}$	generated by Algorithm 2 converges
quadratically to $x^*$, the unique solution of \eqref{eq:p3}.
\end{theorem}
\begin{proof}
  Let $M:=\norm{\nabla^2f(x^*)^{-1}}$. For any $M'>M$ there exists $r_0>0$
  such that
  \begin{align*}
    x\in B(x^*,r_0)\Longrightarrow \nabla^2f(x)\text{ is non-singular, }
    \norm{\nabla^2f(x)^{-1}}\leq M'.
  \end{align*}
  Since $\{f(x_k)\}$ converges to $f(x^*)$, it follows from assumptions
  AS1 and AS3r that $x_k\to x^*$ as $k\to\infty$; therefore, there
  exists $k_0$ such that
  \begin{align*}
    \norm{x^*-x_k} <r_0 \  \text{ for }k\geq k_0.
  \end{align*}
  Define, for $k> k_0$, $s_k$, $s_k^{\mathrm{N}}$, and $s^*_k$ as
   \begin{align*}
     s_k=-(I+\lambda_k\nabla^2f(x_{k-1}))^{-1}\lambda_k\nabla f(x_{k-1}),\;\;
     s_k^{\mathrm{N}}=-\nabla^2f(x_{k-1})^{-1}\nabla f(x_{k-1}),\;\;
     s^*_k=x^*-x_{k-1}.
   \end{align*}
   Observe that $s_k$ is the step of Algorithm 2 at $x_{k-1}$, and
   $s_k^{\mathrm{N}}$ is Newton's step for \eqref{eq:p3} at $x_{k-1}$.
   Define also
   \begin{align*}
     w_k=\nabla^2 f(x_{k-1})(s^*_{k})+\nabla f(x_{k-1})
     =\nabla^2 f(x_{k-1})(x^*-x_{k-1})+\nabla f(x_{k-1}).
   \end{align*}
  Since $\nabla f(x^*)=0$, it follows from assumption AS2 that
  $\norm{w_k}\leq L\norm{s_k^*}^2/2$. Hence
  \begin{align}\label{quadra-conv1}
    \norm{s^*_k-s_k^\mathrm{N}}=\Norm{\nabla^2 f(x_{k-1})^{-1}w_k}\leq
    \dfrac{M'L}{2}\norm{s^*_k}^2.
  \end{align}
Let us now observe that
 $$\norm{s_k}\leq\norm{s_k^\mathrm{N}}.$$
This is a direct consequence of the definition of $s_k$, $s_k^{\mathrm{N}}$, and the monotonicity property  of $\nabla^2 f(x_{k-1})$.
By the two above relations, and the triangle inequality we deduce that 
   \begin{align*}
     \norm{s_k} \leq \norm{s^*_k}+\norm{s_k^{\mathrm{N}}-s^*_k}
     \leq \norm{s^*_k}\left(1+\dfrac{M'L}{2}\norm{s^*_k}\right) .
   \end{align*}
   The first inequality in \eqref{eq:ls.n} is, in the above notation,
   $2\sigma_\ell/L \leq \lambda_k\norm{s_k}$. Therefore,
  \begin{align}\label{basic-lambda}
    \lambda_k^{-1}\leq \dfrac{L}{2\sigma_\ell}\norm{s_k} .
  \end{align}
  It follows from the above definitions that 
  \begin{align*}
    \nabla^2f(x_{k-1})s_k+\lambda_k^{-1}s_k+\nabla f(x_{k-1})=0,\qquad
    \nabla^2f(x_{k-1})s^{\mathrm{N}}_k+\nabla f(x_{k-1})=0 .
  \end{align*}
  Hence $\nabla^2f(x_{k-1})(s^{\mathrm{N}}_k-s_k)=\lambda_k^{-1}s_k$, which gives, by (\ref{basic-lambda})
  \begin{align}\label{quadra-conv2}
    \norm{s_k^{\mathrm{N}}-s_k}\leq M'\lambda_k^{-1}\norm{s_k}
    \leq \dfrac{M'L}{2\sigma_\ell}\norm{s_k}^2 .
  \end{align}
Combining (\ref{quadra-conv1}) with (\ref{quadra-conv2}), we finally obtain
   \begin{align*}
     \norm{x^*-x_k}=\norm{s^*_k-s_k}
     &\leq
       \norm{s^*_k-s^\mathrm{N}_k}+\norm{s^{\mathrm{N}}_k-s_k}\\
     &\leq \dfrac{M'L}{2}\left[\norm{s^*_k}^2+\dfrac{1}{\sigma_\ell}
       \norm{s_k}^2\right]\\
     & =\dfrac{M'L}{2}\left[1+\dfrac{1}{\sigma_\ell}
       \left(1+\dfrac{M'L}{2}\norm{s^*_k}\right)^2
       \right]\norm{x^*-x_{k-1}}^2 .
   \end{align*}

\end{proof}

\section{Concluding remarks} The proximal point method is a basic block
of several algorithms and splitting methods in optimization, such as
proximal-gradient methods, Gauss-Seidel alternating proximal
minimization, augmented Lagrangian methods. Among others, it has been
successfully applied to sparse optimization in signal/image, machine
learning, inverse problems in physics, domain decomposition for PDE'S...
In these situations, we are faced with problems of high dimension, and
this is a crucial issue to develop fast methods. In this paper, we have
laid the theoretical foundations for a new fast proximal method. It is
based on a large step condition. For convex minimization problems, its
complexity is $\bigo(\frac{1}{n^2})$, and global quadratic convergence holds in the regular case for the associated proximal-Newton method. It can be considered as a
discrete version of a regularized Newton continuous dynamical system.
Many interesting theoretical points still remain to be investigated, such as obtaining fast convergence results for maximal monotone operators which are not subdifferentials, the combination of the method
with classical proximal based algorithms, and duality methods, as mentioned above. The
implementation of the method on concrete examples is a subject for
further research.

\appendix

\section{Appendix}
\subsection{A discrete differential inequality}
\begin{lemma}
  \label{lm:tch2}
	Let $\{a_k\}$ be a sequence of non-negative real numbers and let $\tau\geq 0$
	be such that $\tau\sqrt{a_0}\leq 1$.
	If $a_k \leq a_{k-1}-\tau a_{k-1}^{3/2}$ for all $k\geq 1$, then
	\begin{align*}
    a_k\leq\dfrac{a_0}{ \left[1+k\tau \sqrt{a_0}/2\right]^2}.
  \end{align*}
 \end{lemma}

\begin{proof}
  Since $\{a_k\}$ is non-increasing, it follows that $a_k=0$ implies
  $a_{k+1}=a_{k+2}=\cdots=0$ and, consequently, the desired inequality holds for all $k'\geq
  k$. Assume now that $a_k>0$ for some $k\geq 1$.
	Using the assumptions on $\{a_k\}$
	we find the following inequality:
	\begin{align*}
    \dfrac{1}{a_j}\geq \dfrac{1}{a_{j-1}-\tau a_{j-1}^{3/2}}>0\quad \forall j\leq k.
  \end{align*}
	Taking the square root on both sides of latter inequality and using the convexity
  of the scalar function $t\mapsto 1/\sqrt{t}$ we conclude that
	\begin{align*}
    \dfrac{1}{\sqrt{a_j}}\geq \dfrac{1}{\sqrt{a_{j-1}-\tau a_{j-1}^{3/2}}}\geq
    \dfrac{1}{\sqrt{a_{j-1}}}+\dfrac{1}{2a_{j-1}^{3/2}}\tau a_{j-1}^{3/2}
    = \dfrac{1}{\sqrt{a_{j-1}}}+\dfrac{\tau}{2}\quad \forall j\leq k.
  \end{align*}
	Adding the above inequality for $j=1,2,\dots,k$ we obtain
  \[
    \dfrac{1}{\sqrt{a_k}}\geq \dfrac{1}{\sqrt{a_0}}+k\tau /2,
  \]
	which in turn gives the desired result.
  \end{proof}

\subsection{Some examples}
\label{sec:ex}

Consider some simple examples where we can explicitly compute the solution
$(x, \lambda)$ of the algebraic-differential system (\ref{eq:edomm}), and verify that this is effectively a well-posed system.

\paragraph{Isotropic linear monotone operator} Let us start with the following simple situation. Given $\alpha >0$ a positive constant, take ${A} =\alpha I$, i.e.,  for every $x \in {\mathcal H}$
$
{A}x = \alpha x.
$
One  obtains
\begin{align}\label{eq:ex1}
&(\lambda {A}+I)^{-1}x = \frac{1}{1 + \lambda \alpha}x\\
& x- (\lambda {A}+I)^{-1}x = \frac{\lambda \alpha}{1 + \lambda \alpha}x.
\end{align}
Given $x_0 \neq 0$, the algebraic-differential system (\ref{eq:edomm}) can be written as follows
\begin{align}
  &\dot x (t) + \frac{\alpha\lambda (t)}{1 + \alpha \lambda (t)}x(t) =0,\qquad
  \lambda (t)>0,\qquad	\label{eq:ex2}\\
  &\frac{{\alpha\lambda (t)}^2}{1 + \alpha\lambda (t)}	\norm{x(t)}=\theta, \label{eq:ex21}\\
  &x(0)=x_0 \label{eq:ex23}.
\end{align}
Let us integrate the linear differential equation (\ref{eq:ex2}).  Set
\begin{equation}
\Delta (t) :=	\int_0^t    \frac{\alpha\lambda (\tau)}{1 + \alpha\lambda (\tau)} d\tau.
\end{equation}
We have
\begin{equation}
x(t)=	e^{-\Delta (t)} x_0.
\end{equation}
Equation (\ref{eq:ex21}) becomes
\begin{equation}
\frac{{\alpha\lambda (t)}^2}{1 +\alpha \lambda (t)}  e^{-\Delta (t)}	= \frac{\theta}{\norm{x_0}}.
\end{equation}
First, check  this equation  at time $t=0$.  Equivalently
\begin{equation}
\frac{{\alpha\lambda (0)}^2}{1 + \alpha\lambda (0)}   = \frac{\theta}{\norm{x_0}}.
\end{equation}
This equation  defines uniquely $\lambda (0)>0$, because the function $\xi \mapsto \frac{\alpha{\xi}^2}{1 + \alpha \xi}$ is strictly increasing from $[0, +\infty[$ onto $[0, +\infty[$.
Thus, the only thing we have to prove is the existence of a positive function $t \mapsto \lambda (t)$ such that
\begin{equation}
h(t):=	\frac{\alpha{\lambda (t)}^2}{1 + \alpha\lambda (t)}  e^{-\Delta (t)}   \quad \mbox{is constant on} \ [0, +\infty[.
\end{equation}
Writing that the derivative $h'$ is identically zero on $[0, +\infty[$, we obtain that $\lambda(\cdot)$ must satisfy 
\begin{equation}
\lambda'(t)(\alpha\lambda (t)+2) - \alpha{\lambda (t)}^2=0.
\end{equation}
After integration of this first-order differential equation, with Cauchy data $\lambda (0)$,  we obtain
\begin{equation}\label{eq:ex27}
\alpha \ln \lambda (t) - \frac{2}{\lambda (t)} = \alpha t + \alpha \ln \lambda (0) - \frac{2}{\lambda (0)}.
\end{equation}
Let us introduce the function $g: \ ]0, +\infty[ \rightarrow \mathbb R$
\begin{equation}
g(\xi) =  \alpha \ln \xi - \frac{2}{\xi}.
\end{equation}
One can easily verify that, as $t$ increases from $0$ to $+\infty$,  $g(t)$ is strictly increasing from $-\infty$ to $+\infty$ .
Thus, for each $t>0$, (\ref{eq:ex27}) has a unique solution $\lambda(t) >0$.  Moreover,  the mapping $t \to \lambda(t)$ is increasing, continuously differentiable,
and  $\lim_{t \to \infty}\lambda(t) = + \infty$.
Returning to (\ref{eq:ex27}), we obtain that $\lambda(t) \approx e^{t}$ as $t \to +\infty$.

\paragraph{Antisymmetric linear monotone operator}
Take ${\mathcal H} = {\mathbb R}^2$ and ${A}$ equal to the rotation centered at the origin and angle  $\frac{\pi}{2}$.
The operator ${A}$ satisfies ${A}^{*} = -{A}$ (anti self-adjoint). This is a model example of a linear maximal monotone operator which is not self-adjoint.
Set $x= (\xi, \eta) \in {\mathbb R}^2$. We have
$$
{A}(\xi, \eta) = (-\eta, \xi).
$$
\begin{align}\label{eq:ex5}
&(\lambda {A}+I)^{-1}x = \frac{1}{1 + {\lambda}^2 }\Big( \xi + \lambda	\eta, \eta - \lambda \xi \Big)	\\
& x- (\lambda {A}+I)^{-1}x = \frac{\lambda}{1 + {\lambda}^2 }\Big( \lambda\xi - \eta, \lambda \eta + \xi \Big).
\end{align}
The condition $\lambda \norm{(\lambda {A}+I)^{-1}x-x}=\theta$ can be reexpressed as
$$
\frac{\lambda^2}{1 + {\lambda}^2 }\norm{  \Big( \lambda\xi - \eta, \lambda \eta + \xi \Big)}=  \theta.
$$
Equivalently
$$
\frac{\lambda^2}{\sqrt{1 + {\lambda}^2} } \sqrt{ \xi^2 + \eta^2  }=  \theta.
$$
Given $x_0 \neq 0$, the algebraic-differential system (\ref{eq:edomm}) can be written as follows
\begin{align}
  &\dot {\xi} (t) + \frac{\lambda(t)}{1 + {\lambda(t)}^2 }\Big( \lambda (t) \xi (t) - \eta (t) \Big) =0,\qquad
  \lambda (t)>0,\qquad	\label{eq:ex50}\\
  &\dot {\eta} (t) + \frac{\lambda(t)}{1 + {\lambda(t)}^2 }\Big( \lambda (t) \eta (t) + \xi(t) \Big) =0,\qquad
  \lambda (t)>0,\qquad	\label{eq:ex51}\\
  &\frac{\lambda(t)^2}{\sqrt{1 + {\lambda(t)}^2} } \sqrt{ \xi (t)^2 + \eta (t) ^2  }=  \theta, \label{eq:ex52}\\
  &x(0)=x_0 \label{eq:ex53}.
\end{align}
Set $u(t) = \xi (t)^2 + \eta (t) ^2$.
After multiplying  (\ref{eq:ex50}) by $\xi(t)$, and multiplying  (\ref{eq:ex51}) by $\eta(t)$, then adding  the results, we obtain
$$
u'(t) +  \frac{2 {\lambda(t)}^2}{1 + {\lambda(t)}^2}u(t) =0.
$$
Set
\begin{equation}
\Delta (t) :=	\int_0^t    \frac{2 {\lambda(\tau)}^2}{1 + {\lambda(\tau)}^2} d\tau.
\end{equation}
We have
\begin{equation}
u(t)=	e^{-\Delta (t)} u(0).
\end{equation}
Equation (\ref{eq:ex52}) becomes
\begin{equation}
\frac{\lambda(t)^2}{\sqrt{1 + {\lambda(t)}^2} }  e^{-\frac{\Delta (t)}{2}} = \frac{\theta}{\norm{x_0}}.
\end{equation}
First, check  this equation  at time $t=0$.  Equivalently
\begin{equation}
\frac{\lambda(0)^2}{\sqrt{1 + {\lambda(0)}^2} }   = \frac{\theta}{\norm{x_0}}.
\end{equation}
This equation  defines uniquely $\lambda (0)>0$, because the function $\rho \mapsto \frac{{\rho}^2}{\sqrt{1 + {\rho}^2} }$ is strictly increasing from $[0, +\infty[$ onto $[0, +\infty[$.
Thus, the only thing we have to prove is the existence of a positive function $t \mapsto \lambda (t)$ such that
\begin{equation}
h(t):=	\frac{\lambda(t)^2}{\sqrt{1 + {\lambda(t)}^2} }  e^{-\frac{\Delta (t)}{2}}   \quad \mbox{is constant on} \ [0, +\infty[.
\end{equation}
Writing that the derivative $h'$ is identically zero on $[0, +\infty[$, we obtain that $\lambda(\cdot)$ must satisfy
\begin{equation}
\lambda'(t)(2\lambda (t)+ {\lambda(t)}^3) - {\lambda (t)}^3=0.
\end{equation}
After integration of this first-order differential equation, with Cauchy data $\lambda (0)$,  we obtain
\begin{equation}\label{eq:ex55}
\lambda (t) - \frac{2}{\lambda (t)} =  t +  \lambda (0) - \frac{2}{\lambda (0)}.
\end{equation}
Let us introduce the function $g: \ ]0, +\infty[ \rightarrow \mathbb R$
\begin{equation}
g(\rho) =   \rho - \frac{2}{\rho}.
\end{equation}
As $t$ increases from $0$ to $+\infty$,  $g(t)$ is strictly increasing from $-\infty$ to $+\infty$ .
Thus, for each $t>0$, (\ref{eq:ex55}) has a unique solution $\lambda(t) >0$.  Moreover,  the mapping $t \to \lambda(t)$ is increasing, continuously differentiable,
and  $\lim_{t \to \infty}\lambda(t) = + \infty$.
Returning to (\ref{eq:ex55}), we obtain that $\lambda(t) \approx t$ as $t \to +\infty$.

 \if {
\paragraph{Subdifferential of the indicator function of a closed convex set}
Take ${A}= \partial \delta_C$  be equal to the subdifferential $\delta_C$ of the indicator function of a closed convex set $C \subset {\mathcal H}$.
Equivalently ${A}$ is the normal cone mapping to $C$.
Classically
$$
(\lambda {A}+I)^{-1}x = \mbox{proj}_C x.
$$
Note that, in this very special situation, the resolvent operator is independent of $\lambda$.
Given $x_0 \notin C$, the algebraic-differential system \eqref{eq:edomm} can be written as follows
\begin{align}
  &\dot x (t) + x(t)- \mbox{proj}_C x (t) =0,	    \label{eq:ex60}\\
  &\lambda (t)	\norm{ x(t) - \mbox{proj}_C x (t) }=\theta, \label{eq:ex61}\\
  &x(0)=x_0 \label{eq:ex62}.
\end{align}
The operator $Ax= x -\mbox{proj}_C x$ is maximal monotone cocoercive (the identity minus a contraction), hence demipositive.
Thus, \eqref{eq:ex60} is relevant of the theory of semigroups generated by maximal monotone demipositive  operators, see \cite{Bruck}.
Hence there exists a unique orbit $x$ which satisfies \eqref{eq:ex60} and \eqref{eq:ex62}. Moreover $x$ satisfies:
\begin{align}
&\int_0^{\infty}  \| \dot x (t)\|^2 dt < + \infty \\
&t \mapsto \|\dot x (t)\|  \  \  \mbox{is decreasing}
\end{align}
Hence, $\|\dot x (t)\|$ (strictly) decreases to zero. By \eqref{eq:ex60}, this implies that $\| x(t)- \mbox{proj}_C x (t)\|$ (strictly) decreases to zero. Finally
\eqref{eq:ex61} is satisfied by taking
$$
\lambda (t) = \frac{\theta}{\| x(t)- \mbox{proj}_C x (t)\|}.
$$
Clearly, $\lambda (t)$ is uniquely defined, and tends increasingly to $+\infty$ as $t$ goes to infinity.
}
\fi






\end{document}